\documentclass[aihp]{imsart}

\RequirePackage[OT1]{fontenc}

\startlocaldefs

\usepackage[english]{babel}
\usepackage[numbers,comma,square,sort&compress]{natbib}
\usepackage{amsmath,amssymb,amsthm,mathrsfs,enumitem}
\usepackage{xcolor}

\definecolor{my-link}{rgb}{0.5,0.0,0.0}
\definecolor{my-blue}{rgb}{0.0,0.0,0.6}
\definecolor{my-red}{rgb}{0.5,0.0,0.0}
\definecolor{my-green}{rgb}{0.0,0.5,0.0}
\RequirePackage[colorlinks,urlcolor=my-blue,linkcolor=my-red,citecolor=my-green]{hyperref}

\usepackage{xargs}
\usepackage[colorinlistoftodos,prependcaption,textsize=tiny,textwidth=.9in]{todonotes}
\newcommandx{\addmath}[2][1=]{\todo[linecolor=red,backgroundcolor=red!25,bordercolor=red,#1]{#2}}
\newcommandx{\fixtext}[2][1=]{\todo[linecolor=blue,backgroundcolor=blue!25,bordercolor=blue,#1]{#2}}
\newcommandx{\note}[2][1=]{\todo[linecolor=yellow,backgroundcolor=yellow!25,bordercolor=yellow,#1]{#2}}

\usepackage{soul}

\allowdisplaybreaks[1]

\newtheorem{theorem}{Theorem}[section]

\newtheorem{lemma}[theorem]{Lemma}
\newtheorem{corollary}[theorem]{Corollary}

\theoremstyle{definition}
\newtheorem{definition}[theorem]{Definition}
\newtheorem{example}[theorem]{Example}

\theoremstyle{remark}
\newtheorem{remark}[theorem]{Remark}

\numberwithin{equation}{section}

\newcommand{\E}{\mathbb{E}}
\newcommand{\cG}{\mathcal{G}}

\newcommand{\cA}{\mathcal{A}}
\newcommand{\cC}{\mathcal{C}}
\newcommand{\cE}{\mathcal{E}}
\newcommand{\cP}{\mathcal{P}}
\newcommand{\cK}{\mathcal{K}}
\newcommand{\cI}{\mathcal{I}}

\renewcommand{\P}{\mathbb{P}}
\newcommand{\R}{\mathbb{R}}
\newcommand{\Q}{\mathbb{Q}}
\newcommand{\Z}{\mathbb{Z}}
\newcommand{\N}{\mathbb{N}}

\newcommand{\e}{\varepsilon}
\newcommand{\kS}{\mathfrak{S}}
\newcommand{\Lapp}{\Lambda_{\mathrm{pp}}}
\newcommand{\Lapl}{\Lambda_{\mathrm{pl}}}
\newcommand{\gpp}{\Lapp^\infty}
\newcommand{\gpl}{\Lapl^\infty}

\def\w{\omega}

\def\range{\mathcal R}
\def\rangez{\range_0}
\def\rangegen{\range'}
\def\rangegengen{\range''}
\def\Uset{\mathcal U}
\def\Usetz{\Uset_0}
\def\Usetgen{\Uset'}

\font \mymathbb = bbold10 at 11pt
\newcommand{\one}{\mbox{\mymathbb{1}}}    

\providecommand{\abs}[1]{\vert#1\vert}

\newcommand{\fl}[1]{\lfloor{#1}\rfloor} 
\newcommand{\ce}[1]{\lceil{#1}\rceil}

 \newcommand{\cL}{\mathcal{L}}
\def\ZLpp{a} 
 \newcommand{\xhat}{\hat{x}}
 \DeclareMathOperator{\spn}{span}
\newcommand{\Zd}{\Z^d} 
\newcommand{\Qd}{\Q^d} 
 
 \newcommand{\cplus}{\cC_+}
 \newcommand{\gplus}{\cG_+}
 
\def\ri{\mathrm{ri\,}}

\endlocaldefs

\begin{document}

\begin{frontmatter}

\title{A shape theorem and a variational formula for\\ the quenched Lyapunov exponent of\\ random walk in a random potential}
\runtitle{Shape theorem and  variational formula for RWRP}

\begin{aug}
\author{\fnms{Christopher} \snm{Janjigian}\thanksref{t1}\ead[label=e1]{janjigia@math.utah.edu}}, 
\author{\fnms{Sergazy} \snm{Nurbavliyev}\thanksref{t2}\ead[label=e2]{sergazy@math.utah.edu}}
\and
\author{\fnms{Firas} \snm{Rassoul-Agha}\thanksref{t3}\ead[label=e3]{firas@math.utah.edu}}

\thankstext{t1}{Research partially supported by National Science Foundation grant DMS-1954204.}
\thankstext{t2}{Research partially supported by National Science Foundation grants DMS-1811087 and DMS-1715680.}  
\thankstext{t3}{Research partially supported by National Science Foundation grants DMS-1811087 and DMS-1715680.}

\runauthor{C.\ Janjigian, S.\ Nurbavliyev, and F.\ Rassoul-Agha}

\address{$^1$University of Utah, Department of Mathematics, 155 S 1400 E, Salt Lake City, UT 84112. \printead{e1}}
\address{$^2$University of Utah, Department of Mathematics, 155 S 1400 E, Salt Lake City, UT 84112. \printead{e2}}
\address{$^3$University of Utah, Department of Mathematics, 155 S 1400 E, Salt Lake City, UT 84112. \printead{e3}}

\end{aug}

\begin{abstract}
We prove a shape theorem and derive a variational formula for the limiting quenched Lyapunov exponent and the Green's function of 
random walk in a random potential on a square lattice of arbitrary dimension and with an arbitrary finite set of steps. The potential is a function of a stationary environment and the step of the walk. 
This potential is subject to a moment assumption whose strictness is tied to the mixing of the environment.
Our setting includes directed and undirected polymers, random walk in static and dynamic random environment, and, when the temperature is taken to zero, our results also give a shape theorem and a variational formula for the time constant of both site and edge directed last-passage percolation and standard first-passage percolation. 
\end{abstract}




\end{frontmatter}

\section{Introduction}
The model of a random walk in a random
potential (RWRP) on the lattice contains as special cases a large number of other models
frequently studied in the probability literature, including directed polymers in random environments,
random walks in both static and dynamic random environments, and directed
and undirected first-passage percolation. In each of these models, substantial attention has been paid to a quantity which serves as the leading order asymptotic of the model. This quantity goes by various names in different models, including the time constant, the limit shape, the asymptotic free energy, the Green's function, and the quenched Lyapunov exponent. 

In the present paper, we consider the quenched Lyapunov exponent and the exponential rate of decay of the Green's function of a random walk in a random potential with general steps on the lattice and in a stationary potential which is allowed to depend both on the position and on the step of the walk, subject to certain moment conditions. 
Throughout this paper, the Lyapunov exponent gives the leading order exponential decay rate of the partition function for the model where the walk is killed on first reaching a set, while the Green's function corresponds to an unrestricted path length model without killing. We consider these models both in positive temperature and at zero temperature and so our results hold for all of the models described in the previous paragraph. These connections are discussed near the beginning of 
Section \ref{sec:setmain}.

The main contributions of this paper are twofold. First, we show a locally-uniform version of the limits defining the Lyapunov exponent and Green's function exponential decay rate. Such results are sometimes known as \textit{shape theorems} in the literature and play a key role in relating the structure of the Lyapunov exponents or Green's function decay rate to the pre-limit behavior of the model. The uniformity is needed because one frequently needs to apply the limit along random sequences of lattice paths (for example, along a geodesic).

Our second main contribution is to obtain a variational representation for the Lyapunov exponent and Green's function decay rate in terms of random cocycles, which generalizes the variational formulas that were previously obtained for some specializations of the model we consider.
In addition to offering a tool that allows us to study the structure of the Lyapunov exponent or Green's function decay rate generally, the random fields which appear in these variational problems are intimately connected to the prelimit structure of the model. We will mention some of these connections when we review the previous work on the problems we consider.

There are two main technical novelties in this paper. First, the admissible steps of our reference walk are allowed to be any finite subset of $\Z^d$. Most of the previous work focused on nearest-neighbor steps or on restricted-length paths and the previous arguments often rely on these assumptions in essential ways. Second, in most of our results, our reference measure is only required to be shift-stationary rather than shift-ergodic or, as is typical in much of the literature, even i.i.d. 
This introduces a few technical difficulties which can mostly be addressed with standard tricks, but some care is required. Measure-theoretic issues make it non-trivial to rely on the ergodic decomposition to obtain the stationary case from the ergodic case in many of our results and, in fact, we avoid arguments of this type for this reason. This extension is an important generalization because, as in \cite{Dam-Han-14} and \cite{Jan-Ras-20-aop} for example, one frequently works on extensions of the original probability space which are \textit{a priori} only shift stationary. The fact that previous variational formulas have assumed ergodicity needed to be worked around for example in the proof of the key Theorem 4.8 of \cite{Jan-Ras-20-aop}. Our long-term goal is to generalize and give a unified treatment of a number of recent advances made in the study of random polymers and percolation models
to as wide a class as possible. The present work is a necessary first step toward extending these connections to the general RWRP setting.

\subsection*{Previous work}
First-passage percolation (FPP) was introduced in 1965 by Hammersley and Welsh \cite{Ham-Wel-65} to model the flow of water through a porous medium. It was the first in a long line of growth models which have been central to the development of modern probability. One of the main questions considered in the early work on such processes was whether the rescaled cluster has a limit shape and, if so, how strong of a limit can be expected to hold. Richardson proved the first major result of this type in 1973 in \cite{Ric-73}, essentially corresponding to the case of i.i.d.~geometric edge weights and showed that the cluster converges as a set or, equivalently, in the local-uniform sense which is of interest to us in the present paper. 
Such results are known as \textit{shape theorems}. Around the same time, Kingman \cite{Kin-73} first proved the subadditive ergodic theorem, in part motivated by the problem of showing the existence of the limiting time constant (or limit shape) in percolation. In 1981, Cox and Durrett \cite{Cox-Dur-81} gave necessary and sufficient conditions for the existence of the limit shape in several modes of convergence when the weights are i.i.d., including pointwise in probability, pointwise almost surely, and locally-uniformly. Durrett and Liggett \cite{Dur-Lig-81} subsequently proved regularity properties and the existence of a flat edge in the limit shape of Richardson's model when the probability that an edge is open is sufficiently close to $1$.

The directed polymer model, a precursor of the random walk in a random potential (RWRP), was introduced in the physics literature in 1985 by Huse and Henley \cite{Hus-Hen-85} to model the domain wall in the ferromagnetic Ising model with random impurities. They were taken up in the mathematics literature in \cite{Imb-Spe-88}. This model is a measure on random paths interacting with a random environment and it can be viewed as a noisy version of percolation. Through the Feynman-Kac representation, the partition function in these models can frequently be viewed as the solution to a random Schr\"odinger equation. In this setting, the limiting free energy or quenched Lyapunov exponent is the leading-order quantity analogous to the limit shape. This quenched Lyapunov exponent, viewed as a function of the direction, also serves as the convex dual of the large deviation rate function for the endpoint of the path under the quenched path measure. 

In 1988, Schroeder \cite{Sch-88} obtained the leading order exponential decay rate of the Green's function for the Schr\"odinger operator $-\Delta + V$ with $V(x)$ a periodic continuous function on $\R^d$ using large deviation techniques for diffusions originally developed by Donsker and Varadhan. Subsequently, in 1994, Sznitman \cite{Szn-94} considered the model of a $d$-dimensional Brownian motion in a smoothed Poissonian potential, a continuum random walk in a random potential. He proved a quenched large deviation principle as well as a locally-uniform shape theorem for the Green's function corresponding to the operator $(1/2)\Delta - (\lambda + V)$, where $V(x)$ is the random potential and $\lambda$ is a non-negative constant. 

In \cite{Zer-98-aap} and \cite{Zer-98-aop} Zerner, motivated by ideas in \cite{Szn-94}, showed the existence of the quenched Lyapunov exponent and a corresponding locally-uniform shape theorem for two models of a random walk in a random potential on $\Z^d$: a random walk in a random environment (RWRE) with nearest neighbor steps on $\Z^d$, where the logarithm of the i.i.d.\ transitions have at least $d$ finite absolute moments and the simple symmetric random walk in a non-negative i.i.d.\ site potential with the same moment condition. In the latter model, which describes a simple symmetric random walk on $\Z^d$ in an i.i.d.\ potential $V(x)$, the quantity of interest is again the Green's function for the operator $\Delta - V$ on $\Z^d$. In \cite{Zer-98-aop}, Zerner also proved a quenched large deviation principle for the RWRE model, but under the additional assumption of the \textit{nestling} condition. This is the condition that zero lies in the convex hull of the support of the law of the drift of the walk. 

The duality between velocity and shifts of the potential, familiar from large deviation theory, plays an important role in the present work. This duality was first observed in the context of random polymer measures by Zerner in his works \cite{Zer-98-aap,Zer-98-aop,Zer-00}.

The problem of proving the quenched large deviation principle in the non-nestling case remained open until Varadhan  \cite{Var-03-cpam} proved the result in the case where the walk has finitely many possible steps and the transition probabilities of all nearest-neighbor steps are uniformly bounded away from zero. A few years later, Flury \cite{Flu-07} proved the quenched large deviation principle for nearest-neighbor random walk in an i.i.d.\ random potential with $d$ finite absolute moments.

In \cite{Mou-12}, Mourrat studied the same model as was previously considered by \cite{Zer-98-aap,Flu-07} and, using a modification of the arguments in \cite{Cox-Dur-81} along with a renormalization scheme, proved necessary and sufficient conditions for the existence of the Lyapunov exponent in several modes of convergence when the weights are i.i.d., including in $L^1$, in probability, and almost surely. This work and \cite{Cox-Dur-81} are notable for allowing $V(x) = \infty$ and so include walks on percolation clusters.

The first papers to consider RWRP at the level of generality considered in the present paper were \cite{Ras-Sep-14} and \cite{Ras-Sep-Yil-13}, which showed the existence of the Lyapunov exponent in the point-to-point and point-to-hyperplane geometries for fairly general restricted path length RWRP models on the lattice. In these works, the reference walk is allowed to take a given number of steps from a finite subset of the lattice and the potential is allowed to depend both on the location of the walk and the increment of the walk. In order to work at this level of generality, the potential is required to satisfy certain mixing and moment conditions, the strictness of which varies. For example, if the weights are bounded, then the mixing condition reduces to ergodicity, while if the potential is only known to have $d+\e$ moments, the mixing condition essentially reduces to independence.  
In the current work, we consider a similar general setting but with different path geometries. Here we either run the walk until the first time it reaches a site or level, otherwise known as running a random walk with killing, or else consider all paths which reach a given site or level. The connection between restricted and unrestricted path length models was recently exploited in \cite{Kri-Ras-Sep-20-fluffy-} to derive information about the asymptotic length of the geodesic path in FPP.

A common issue encountered in the study of models of this type is that while we can show the limit shape exists and satisfies certain soft properties like homogeneity and convexity, it is difficult to go further than that. For this reason, among others, it is valuable to have a variational characterization of the limit shape in terms of (typically infinite-dimensional) observables of the model. The main ideas which led to the development of such variational formulas for RWRP first arose in the context of homogenization of Hamilton-Jacobi equations in \cite{Kos-Rez-Var-06} and \cite{Kos-Var-08}. 
These ideas were adapted in \cite{Ros-06} to give two variational formulas for the level-1 quenched large deviation rate function for undirected RWRE. 
One of the formulas is through the familiar convex duality with entropy and the other formula is in terms of objects called cocycles.  
These formulas were extended to the (two-step) level-2 large deviation rate function in \cite{Yil-09-cpam}, to the level-3 rate function in \cite{Ras-Sep-11}, and then to the case of directed and undirected restricted path length RWRP and percolation models in \cite{Ras-Sep-Yil-13,Ras-Sep-14,Geo-Ras-Sep-16}. 
In the case of the standard FPP model, a related variational formula in terms of cocycles was derived in \cite{Kri-16} and the (level 1) entropy formula was recently proved in \cite{Bat-20-}, where the formula was utilized to answer some questions about asymptotic properties of geodesics.
The entropy variational formula was also proved for the Green's function decay rate for the Schr\"odinger operator with a periodic potential in 
the aforementioned \cite{Sch-88} and this was extended to the case of a more general random ergodic potential in \cite{Rue-14}. In \cite{Rue-16}, this variational expression was used to prove regularity properties of the Lyapunov exponent, as a function of the potential and of the law of the environment. Another variational formula for the limiting free energy in directed polymer models appeared in \cite{Bat-18,Bat-Cha-20} and was used to study localization properties.
In the current paper, we prove a cocycle variational formula for the Lyapunov exponent and the decay rate of the Green's function.
When specialized to the case of the nearest-neighbor FPP model, the variational formula that we prove will appear in \cite{Kri-Ras-Sep-20-puppy-}. 

The extremizing objects in the cocycle variational formula are known as correctors, by way of analogy to the Hamilton-Jacobi setting in which they originally arose. Correctors were initially defined as first-order terms in the small-$\e$ expansion in the homogenization of a Hamilton-Jacobi equation. More generally, they are cocycles satisfying certain conditions of compatibility with the random environment. These extremizing correctors encode much of the large-scale information in the model. For example, it was observed in \cite{Ras-Sep-Yil-17-ber} that they can be used to characterize the weak, strong, and very strong disorder regimes of directed polymer models. In the same vein, cocycles satisfying the compatibility condition arise naturally in the context of RWRE, RWRP, and percolation as the (analogues of) Doob $h$-transforms that one obtains by conditioning the random path to have a different law of large numbers behavior than is typical. This connection first appears in \cite{Yil-09-aop} and \cite{Yil-11-aop}. 
Through this connection one can use these correctors to construct infinite path length limits (Gibbs measures and infinite geodesics) and study their properties. See the recent papers: \cite{Bal-Ras-Sep-19}, \cite{Dam-Han-14}, \cite{Dam-Han-17}, \cite{Fan-Sep-18-}, \cite{Geo-Ras-Sep-17-ptrf-2},  \cite{Geo-Ras-Sep-17-ptrf-1}, \cite{Geo-etal-15}, \cite{Hof-08}, \cite{Jan-Ras-20-aop},\cite{Jan-Ras-20-jsp}, \cite{Jan-Ras-Sep-19-}. 
As mentioned above, in Theorem 4.8 of \cite{Jan-Ras-20-aop} it was noticed that variational problems of the type we produce here can be used to resolve a key technical obstruction in the construction of the cocycles which are needed in order to build these infinite volume objects. 
In some special (solvable) cases of two-dimensional directed RWRE and RWRP models it was shown in \cite{Geo-etal-15,Bal-Ras-Sep-19} that these $h$-transforms manifest Kardar-Parisi-Zhang (KPZ) fluctuation behavior.
\medskip

\textbf{Acknowledgements.} The authors thank Erik Bates and Timo Sepp\"al\"ainen for helpful comments.


\section{Setting, notation, and main results} \label{sec:setmain}
Throughout the paper $(\Omega,\kS)$ will denote a Polish space endowed with its Borel $\sigma$-algebra. 
A sample point $\w$ in $\Omega$ is called an {\it environment}.
We assume this measurable space is equipped with a group of measurable commuting bijections $T=\lbrace T_z:\Omega\to \Omega:z\in \Zd \rbrace$, i.e.\ $T_0$ is the identity map and $T_x\circ T_y=T_y\circ T_x=T_{x+y}$. 
We are given a probability measure $\P$ on $(\Omega,\kS)$ that is invariant under $T_z$ for all $z\in\Zd$.
Expectation with respect to $\P$ is denoted by $\E$. 
 For a subset $\rangegen\subset\Z^d$ we say $\P$ is ergodic under the shifts $\{T_z:z\in\rangegen\}$ if $\P(A)\in\{0,1\}$ for every event $A$ such that $T_z^{-1} A=A$ for all $z\in\rangegen$.

We will denote the set of real numbers by $\R$, the rational numbers by $\Q$, and the set of whole numbers by $\Z$. A $+$ subscript indicates nonnegative numbers.

Let $\range$ be a finite subset of $\Z^d$ with at least two points in it. 
We are given a measurable function $V : \Omega \times \range \to \R$, which we call a {\it potential}. 

Let $p:\range\to(0,1)$ be a probability kernel, i.e.\ $\sum_{z\in\range}p(z)=1$. For $x\in\Z^d$, let $P_x$ denote the distribution of the time-homogeneous random walk with transition kernel $p$ and starting point $x$. $E_x$ denotes the corresponding expectation. The random walk itself is denoted by $\{X_n:n\in\Z_+\}$.
A sequence $(a_i)_{i=m}^n$ is denoted by $a_{m:n}$ and similarly for $a_{m:\infty}$, $a_{-\infty:n}$, and $a_{-\infty:\infty}$.
A sequence $x_{m:n}$ or $x_{m:\infty}$ with $x_m=x$ and $x_{i+1}-x_i\in\range$ for all $i$ is called an {\it admissible path}. For such a path we use $z_i=x_i-x_{i-1}$ to denote the increments. When the sequence is the random walk itself we use
$X_{m:n}$ and its increments are denoted by $Z_{m+1:n}$.

For $y\in \Z^d$  let $\tau_y$ be the time of first return of the walk to site $y$:
		\[ \tau_y=\inf \lbrace n\ge 1 : X_{n}=y \rbrace\] 
with the convention that $\inf \varnothing=\infty$. 

Throughout, $\abs{x}_1$ denotes the $\ell^1$-norm on $\R^d$ and for $\xi\in\R^d$ and  $A,B\subset\R^d$
	\[\text{dist}(x,A)=\inf\{\abs{y-x}_1:y\in A\}\quad\text{and}\quad\text{dist}(A,B)=\inf\{\abs{y-x}_1:x\in A,y\in B\}.\] 
 For $\rangegen\subset \Z^d,$ let 
	\begin{align*}
	&\cG(\rangegen)=\Bigl\{ \sum_{z\in \rangegen}b_z z: b_z\in \Z  \Bigr\},&\hspace{-30pt}
	\gplus(\rangegen)=\Bigl\{ \sum_{z\in \rangegen}b_z z: b_z\in \Z_+  \Bigr\},\\
	&\cplus(\rangegen)=\Bigl\{ \sum_{z\in \rangegen}b_z z: b_z\in\R_+  \Bigr\},&\hspace{-30pt}\text{and}\quad
	\cplus'(\rangegen)=\Bigl\{ \sum_{z\in \rangegen}b_z z: b_z\in \Q_+ \Bigr\},
	\end{align*}
be respectively the additive group, additive semigroup, cone, and rational cone, generated by $\rangegen$.
We write $\cG$, $\gplus,\cplus$ and $\cplus'$ when $\rangegen=\range$.

For $x,y\in\Z^d$ with $y-x\in \gplus\setminus \{0\}$ let
		\begin{align}\label{a-def}
		a(\w,x,y)=-\log E_{x}\Bigl[ \exp\Bigl\{-\sum_{k=0}^{\tau_y-1} V(T_{X_k}\w, Z_{k+1})\Bigr\} \one_{\lbrace \tau_y<\infty \rbrace}  \Bigr]
		\end{align}
and set $a(\w,x,x)=0$. As is customary, we typically omit the $\w$ and write $a(x,y)$. Describing the asymptotic growth of 
$a(0,x)$ as $\abs{x}_1\to\infty$ is the main goal of this paper.\medskip

The following is a list of examples covered by our general setting.

\begin{example}[Product environment]\label{ex:product}
A natural  choice of $\Omega$ is a product space $\Omega=\Gamma^{\Z^d}$,
where $\Gamma$ is a Polish space, equipped with the product topology, Borel $\sigma$-algebra $\kS$,  generic points $\w=(\w_x)_{x\in\Z^d}$, and  translations $(T_x\w)_y=\w_{x+y}$. Here
  $\P$ is an {\it i.i.d.}\ or {\it product measure} if 
  the coordinates $\{\w_x:x\in\Z^d\}$ are independent and identically distributed  (i.i.d.) random
  variables under $\P$. 
  We say that $\P$ has a {\it finite range of dependence} if there exists an $M\ge0$ such that for any subsets $A,B\subset\Z^d$ with $\abs{x-y}_1>M$ for all $x\in A$ and $y\in B$,  
$\{\w_x:x\in A\}$ and $\{\w_x:x\in B\}$ are independent under $\P$.
The potential $V$ is said to be {\it local} if 
it depends on only finitely many coordinates $\w_x$, i.e.\ if there exists an $L\ge 0$ such that for all $z\in\range$, $V(\w,z)$ is measurable with respect to $\sigma(\{\w_x:\abs{x}_1\le L\})$. 
\end{example}

\begin{example}[Edge and vertex weights]
Random weights assigned to the vertices of $\Z^d$  can be  modeled by $\Omega=\R^{\Z^d}$ and $V(\w)=\w_0$.  In fact, it is sufficient to take $\Omega=\R^{\cG}$ since the coordinates outside $\cG$  are not needed as  long as paths begin at points in $\cG$.  

To represent  directed edge weights  we can  take   $\Omega=\Gamma^{\cG}$ with $\Gamma=\R^\range$ where an element $s\in\Gamma$ represents the weights of the admissible edges out of the origin:  $s=(\w_{(0,z)}: z\in\range)$.    Then $\w_x=(\w_{(x,x+z)}: z\in\range)$ is the vector of edge weights out of vertex $x$. 
 Shifts act by $(T_u\w)_{(x,y)} =\w_{(x+u,y+u)}$ for $u\in\cG$. The potential is $V(\w,z)=\w_{(0,z)}=$ the weight of the edge $(0,z)$.   

To have weights on undirected nearest-neighbor edges take  
$\Omega=\R^{\cE}$ where   $\cE=\{ \{x,y\}\subset\Z^d:  \abs{y-x}=1\}$  
is the set of undirected nearest-neighbor  edges on $\Z^d$. Now   
$\range=\{ \pm e_i: i=1,\dotsc,d\}$, $V(\w,z)=\w_{\{0,z\}}$  and  $(T_u\w)_{\{x,y\}} =\w_{\{x+u,y+u\}}$ for $u\in\Z^d$.  
\end{example}

\begin{example}[Strictly directed walk] \label{ex:dir-iid}
This is the case where  
 $0$ lies  outside the convex hull $\Uset$ of $\range$.
It is equivalent to the existence of  $\hat u\in\Z^d$ such that $\hat u\cdot z>0$ for all $z\in\range$. A familiar special case is the one where $\range=\{e_1,\dotsc,e_d\}$. Another familiar directed polymer is the one with $\range=\{e_1\pm e_2,\dotsc,e_1\pm e_{d+1}\}$.
\end{example}

\begin{example}[Stretched polymer]\label{ex:stretch}  
A stretched polymer has an external field $h\in\R^d$
that biases the walk, so the potential is 
$V(\w,z)=\Psi(\w)+h\cdot z$.  The two cases most studied in the literature are the ones with $\range=\{e_1,\dotsc,e_d\}$ 
and $\range=\{\pm e_1,\dotsc,\pm e_d\}$.
\end{example}

\begin{example}[Random walk in random environment]\label{ex:rwre}
To cover RWRE take $V(\w,z)=-\log \pi_z(\w)+\log p(z)$ where
$(\pi_z)_{z\in\range}$ is a measurable mapping from $\Omega$ into $\cP=\{(\rho_z)_{z\in\range} \in[0,1]^\range:\sum_z \rho_z=1\}$, the space of
probability distributions on $\range$.  The 
{\sl quenched path measure}  $Q^\w_x$ of RWRE
started at $x$ is the probability measure on the path space $(\Z^d)^{\Z_+}$ 
  defined by the initial condition $Q^\w_x(X_0=x)=1$ and 
the transition probability   $Q^\w_x(X_{n+1}=y+z\vert X_n=y)
=\pi_z(T_y\w)$, $z\in\range$.    Then $e^{-a(\w,x,y)}$ is the quenched probability $Q_x^\w(\tau_y<\infty)$ the random walk started at $x$ will ever reach $y$. 
Important cases of RWRE are the {\it space-time case}, where $\range=\{e_1,\dotsc,e_d\}$, and the {\it nearest-neighbor case}, where $\range=\{\pm e_1,\dotsc,\pm e_d\}$.
\end{example}

\begin{example}[Random growth]
Using $\beta V$ in place of $V$, where $\beta\in\R$ is a parameter called the {\it inverse temperature}, and sending $\beta$ to $\infty$ or $-\infty$ one gets, respectively, the models known as {\it first-} and {\it last-passage percolation}.
In the first case, $a(\w,x,y)$ degenerates to
    \begin{align}\label{a-FPP}
    a_\infty(\w,x,y)=\min_{n\in\N}\min_{x_{0:n}}
    \Bigl\{\sum_{k=0}^{n-1}V(T_{x_k}\w, z_{k+1})\Bigr\},
    \end{align}
where the minimum is taken over all admissible paths $x_{0:n}$ from $x_0=x$ to $x_n=y$ that reach $y$ for the first time at step $n$. For last-passage percolation the minimum is replaced by a maximum.

Standard first-passage percolation with edge weights is the case where 
$\range=\{\pm e_1,\dotsc,\pm e_d\}$, $\Omega=\R^{\cE}$, $\cE=\{ \{x,y\}\subset\Z^d:  \abs{y-x}=1\}$, and $V(\w,z)=\w_{\{0,z\}}$.

Directed last-passage percolation with vertex weights is the case where 
$\range=\{e_1,\dotsc,e_d\}$, $\Omega=\R^{\Z^d}$, and $V(\w,z)=\w_0$.
\end{example}

For $p\ge1$ and $f:\Omega\times\range\rightarrow\R$ measurable we write $f\in L^p$ to say that $\E[\abs{f(\w,z)}^p]<\infty$ for all $z\in\range$. This includes the case when $f$ is only a function of $\w$.\smallskip

For $\rangegen\subset\range$ and $z\in\rangegen\setminus\{0\}$, a nonnegative measurable function $g:\Omega\to\R$ is said to be in class $\cL_{z,\rangegen}$ if  
	\begin{equation}\label{def:cL}
	\varlimsup_{\varepsilon\searrow 0}\varlimsup_{n\to \infty} \max_{\substack{x\in\gplus(\rangegen)\\ \abs{x}_1\le n}}\frac{1}{n}\sum_{0\le k\le \varepsilon n}g(T_{x+kz}\w)=0 \quad \text{for}\quad\P \text{-a.e.\ } \w.
	\end{equation}
Membership in $\cL_{z,\rangegen}$ can come from a balance between moments of $g$ and the amount of mixing of $\P$. Bounded $g$ guarantees $g\in \cL_{z,\range}$ for any $\P$ and $ z\in\range\setminus\{0\}$. In the setting of Example \ref{ex:product}, Lemma A.4 of \cite{Ras-Sep-Yil-13} implies that $g\in \cL_{z,\range}$, for any $z\in\range\setminus\{0\}$, if $g$ is local, $\P$ has a finite range of dependence, and $g\in L^p$ for some $p>d$. 
%
%
\smallskip

We now recall a few definitions and facts from convex analysis. A convex subset $A$ of a convex set $K\subset\R^d$ is called a face if for all $\xi \in A$ and $\zeta,\eta\in K$, $\xi=t\zeta+(1-t)\eta$ with $t\in(0,1)$ implies $\zeta,\eta\in A$. The intersection of faces is clearly a face. 
$K$ itself is a face. By \cite[Corollary 18.1.3]{Roc-70} any other face of $K$ is entirely contained in the relative boundary of $K$. Extreme points of $K$ are the
zero-dimensional faces. By \cite[Theorem 18.2]{Roc-70} each point $\xi\in K$ has a unique face $K_\xi$ 
such that $\xi\in\ri K_\xi$, where for a set $A\subset\R^d$, $\ri A$ denotes its relative interior.  
By \cite[Theorem 18.1]{Roc-70} if $\xi\in K$ belongs to a face $A$ then any representation of $\xi$ as a convex combination of elements of $K$ involves only elements of $A$.
\cite[Theorem 18.3]{Roc-70} says that if $K$ in the convex set, respectively the convex cone, generated by a set $R$, then a face $A$ of $K$ is the convex hull, respectively the convex cone, generated by $R\cap A$. (The convex cone generated by the empty set is the singleton $\{0\}$.)

Let $\range_\xi= \range\cap\cC_\xi$ and let $\Uset$ be the convex hull of $\range$. Then $\Uset_\xi=\Uset\cap\cC_\xi$ and $\range_\xi=\range\cap \Uset_\xi$. 
For $\xi\not\in\Uset$ let $\Uset_\xi=\varnothing$. Note that in this case $\Uset\cap\cC_\xi$ and $\range_\xi$ may not be empty.
However, for $\xi=0$, if $0\not\in\Uset$, then $\cC_0=\{0\}$ and $\Usetz=\rangez=\varnothing$.
By \cite[Corollary A.2]{Ras-Sep-Yil-13}, $0\in \Uset$ is equivalent to the existence of loops, i.e.~there exist $ z_i\in\range$, $1\le i\le k$, such that $\sum_{i=1}^kz_i=0$. By \cite[Corollary A.3]{Ras-Sep-Yil-13}, $0\in\ri\Uset$ is equivalent to the existence of an admissible path between any two points $x,y\in\cG$.\smallskip

Given a face $\cA$ of $\cplus$ and $\delta>0$ let
\begin{align}\label{cAdel}
\cA_\delta=\Bigl\{ \xi \in \cA\setminus\lbrace0\rbrace:\text{dist}\Bigl(\frac{\xi}{|\xi|_1},\cA\setminus\ri\cA\Bigr)\ge\delta \Bigr\}.
\end{align}
Note that $\xi\in\cA_\delta\setminus\{0\}$ implies $\xi/|\xi|_1\in \cA_\delta$. 
Recall that $\range_0=\range\cap \Uset_0$, where $\Uset_0$ is the unique face of $\Uset$ that contains $0$ in its relative interior. If $0\not\in\Uset$ then $\Uset_0=\range_0=\varnothing$.
Let $V^+=\max(V,0)$. 

The following is our first main result.

\begin{theorem}\label{shapetheorem}
    Assume $V^+\in L^1$.
     Assume also that one of the following two holds:
	 \begin{enumerate}[label={\rm(\roman*)}, ref={\rm\roman*}] \itemsep=3pt 
	\item\label{ass:V>c} We have
	\begin{equation}\label{V>0} 
	\P\{V(\w,z)\ge 0\}=1 \quad \text{for all }z\in \rangez\quad\text{{\rm(}with possibly $\rangez=\varnothing${\rm)}}
	\end{equation}
	and there exists a $c\in \R$ such that $\P\{V(\w,z)\ge c\}=1$ for all $z\in\range\setminus\rangez$.
	\item\label{ass:iid} The setting of Example \ref{ex:product} is in force and $0\not \in \Uset$, 
	$V$ is local, $\P$ has a finite range of dependence, and $V^+\in L^p$ for some $p>d$.
	\end{enumerate}
	Fix a face $\cA\ne\{0\}$ of $\cplus$ {\rm(}possibly $\cplus$ itself\,{\rm)}.  Let $\rangegen=\range\cap\cA$. Assume $V^+(\w,z)\in \cL_{z,\rangegen}$ for each $z\in\rangegen\setminus\{0\}$ and $V^+(\w,0)\in\cL_{\hat z,\rangegen}$ for some $\hat z\in\rangegen\setminus\{0\}$.
	Then there exist two stochastic processes $\{\alpha(\xi):\xi\in\ri\cA\}$ and $\{\alpha_\infty(\xi):\xi\in\ri\cA\}$ 
	such that the following hold $\P$-almost surely.
	\begin{enumerate}[label={\rm(\alph*)}, ref={\rm\alph*}] \itemsep=3pt 
	\item \label{alpha<Cxi} There is a positive random variable $C$ which satisfies $\E[C] < \infty$ so that for all $\xi\in\ri\cA$, $\abs{\alpha(\xi)}\le C\abs{\xi}_1$. 
	\item For all $s\in\R_+$ and $\xi,\zeta\in\ri\cA$
		\[\alpha(s\xi)=s\alpha(\xi)\quad\text{and}\quad\alpha(\xi)+\alpha(\zeta)\ge\alpha(\xi+\zeta).\]
	Consequently, $\alpha$ is convex and thus continuous on $\ri\cA$.
	\item\label{shapetheorem.b} For all $ \delta>0$ 
	\begin{equation}
	\varlimsup_{ \substack{ \abs{x}_1\to \infty \\ 
			x\in \gplus\cap \cA_\delta} } \dfrac{\abs{ a(0,x)-\alpha(x) }}{\abs{x}_1}= 0. \label{lim:shape}
    \end{equation}
	\item\label{al_infty} Claims \eqref{alpha<Cxi}-\eqref{shapetheorem.b} also hold for $a_\infty$ and $\alpha_\infty$ in place of, respectively, $a$ and $\alpha$.
	\end{enumerate}
	If furthermore $\P$ is ergodic under $\{T_z:z\in\range\cap\cA\}$, then $\alpha$ and $\alpha_{\infty}$ are deterministic on $\ri\cA$, and so is $C$ in \eqref{alpha<Cxi}.
\end{theorem}

\begin{remark}
	If $\P$ is not ergodic under the shifts $\{T_z:z\in \range\cap\cA'\}$ for some face $\cA'$ of $\cA$, 
	then $\alpha$ and $\alpha_\infty$ can be genuinely random on $\cA'$.
\end{remark}

\begin{remark}
If $0\in\ri\Uset$ then $\cC_+$ is the same as the linear span of $\range$ and the only face of $\cC_+$ is $\cC_+$ itself. In this case, the above shape theorem holds on all of $\cC_+$.
\end{remark}

\begin{remark}
In fact, we prove  a slightly stronger version of the above theorem. See Theorem \ref{shapetheorem.aux} below. Inspection of the proof of that theorem shows that the only reason we have $\cA_\delta$ instead of $\cA$ in \eqref{lim:shape} is because  $\alpha$ and $\alpha_\infty$ are guaranteed to be continuous there. Given the continuity of $\alpha$ on all of $\cA\setminus\{0\}$, 
 \eqref{lim:shape} would hold with $\cA_\delta$ replaced by $\cA$. The same is true of $\alpha_\infty$. Following the ideas in the proof of Theorem 3.2 
in \cite{Ras-Sep-14} one should be able to prove the continuity of $\alpha$ and $\alpha_\infty$ on $\cA$ in the setting of  
Example \ref{ex:product}, when $\{(V(T_x\w,z))_{z\in\range}:x\in\cG\}$ are i.i.d.\ under $\P$, $V\in L^p$ for some $p>d$, and $0$ is an extreme point of $\cA$.
\end{remark}

As a corollary of the above theorem, we have the following point-to-point limit. 

\begin{corollary}\label{cor:shape}
Assume the setting of Theorem \ref{shapetheorem}. The following holds $\P$-almost surely: For any $\xi\in\ri\cA$ 
and any sequence $x_n\in \gplus\cap \cA$ such that $x_n/n\to\xi$ as $n\to\infty$
	\begin{align}\label{p2plim}
	\lim_{n\to\infty}\dfrac{a(0,x_n)}n=\alpha(\xi)\quad\text{and}\quad\lim_{n\to\infty}\dfrac{a_\infty(0,x_n)}n=\alpha_\infty(\xi).
	\end{align}
	\end{corollary}

Our second main result gives variational formulas for $\alpha$ and $\alpha_\infty$ in terms of cocycles, which we define next. Let $\cA\ne\{0\}$ be a face of $\cC_+$.
Let $\range'=\cA\cap\range$. Recall that $\cG(\range')$ is the additive subgroup of $\Z^d$ generated by $\range'$.

\begin{definition}
	A measurable function $B:\Omega\times \cG(\range')^2\to \R$ is said to be a cocycle if 
	\begin{align*}
	B(\w,x,y)+B(\w,y,z)=B(\w,x,z)\quad\text{for $\P$-a.e.\ $\w$ and all $x,y,z\in\cG(\range')$}.
	\end{align*}
	$B$ is said to be covariant if
	\begin{align*}
	B(\w,x+z,y+z)=B(T_z\w,x,y)\quad\text{for $\P$-a.e.\ $\w$ and all $x,y,z\in\cG(\range')$}.
	\end{align*}
	$B$ is said to be $L^1$ if $\E[\abs{B(\w,x,y)}]<\infty$ for all $x,y\in\cG(\range')$.
\end{definition}
Let $\cK_\cA$ denote the space of $L^1$ covariant cocycles as defined above. Let $\cI_\cA$ denote the $\sigma$-algebra of events $A\in\kS$ such that $T_x^{-1}A=A$ for all $x\in\cG(\range')$. For $B\in \cK_\cA$, 
$\E[B(0,x)\,|\,\cI_\cA]$ is an additive function of $x\in\cG(\range')$.  Furthermore, by \cite[Proposition P1 on page 65]{Spi-76} the additive group $\cG(\range')$ is linearly isomorphic to a $\Z^k$ for some $k\le d$. 
It follows 
that there exists a random vector $m(B)\in\R^d$ such that $\P$-almost surely,
\[ \E [B(0,x)\,|\,\cI_\cA] = m(B)\cdot x \quad \text{for all }x\in \cG(\range').\]

Note that $m(B)$ is not necessarily unique unless $\range'$ linearly spans $\R^d$, but
the inner products $m(B)\cdot x$ for $x\in\cG(\range')$ are uniquely defined. Set $h(B)=-m(B)$. Define
\[\cK^+_\cA(V) = \Bigl\{ B\in \cK_\cA :   \sum_{{z}\in \range'}p(z) e^{-V(\w,z)-B(\w,0,z)} \le 1 \text{ for }\P\text{- almost every }\w \Bigr\}\]
and
\[\cK^+_{\cA,\infty}(V) = \Bigl\{ B\in \cK_\cA :   \min_{{z}\in \range'}\{V(\w,z)+B(\w,0,z)\}\ge0\text{ for }\P\text{- almost every }\w \Bigr\}.\] 

 Our second main result gives the cocycle variational formula mentioned in the introduction. 

\begin{theorem}\label{VarForm}
	Assume the setting of Theorem \ref{shapetheorem}.
	Then, for any face $\cA\ne\{0\}$ of $\cC_+$ and any $B\in\cK^+_\cA(V)$ we have with $\P$-probability one, for any $\xi\in\cA$ 
	\begin{align}\label{alpha>h.xi}
	\alpha(\xi)\ge h(B)\cdot\xi.
	\end{align}
    Assume further that $\P$ is ergodic under $\{T_z:z\in\range\cap\cA\}$. Then for any $\xi\in\ri\cA$
\begin{align}\label{alpha-var}
	\alpha(\xi) = \sup_{B\in \cK^+_\cA(V)} h(B)\cdot\xi
	\end{align}
and there exists a $B\in\cK^+_\cA(V)$ such that $\alpha(\xi)=h(B)\cdot\xi$. 

The same results hold if $\alpha$ and $\cK^+_\cA$ are replaced, respectively, by $\alpha_\infty$ and $\cK^+_{\cA,\infty}$.
\end{theorem}

\begin{remark}
When $\range=\{\pm e_1,\dots,\pm e_d\}$ a variational formula in terms of cocycles was proved in \cite{Kri-16}. Their formula is different from \eqref{alpha-var}. Our variational formula \eqref{alpha-var} for FPP model with $\range\{\pm e_1,\dotsc,\pm e_d\}$ will appear in the forthcoming \cite{Kri-Ras-Sep-20-puppy-}, where also the connection to the formula in \cite{Kri-16} is worked out.
\end{remark}

\begin{remark}
The proofs of Theorem \ref{shapetheorem}, Corollary \ref{cor:shape}, and Theorem \ref{VarForm} for the case of $a$, defined in \eqref{a-def}, are almost identical to those of the case of $a_\infty$, defined in \eqref{a-FPP}. 
We therefore only present the proofs in the former case, while highlighting where significant changes need to be made if there are any. \end{remark}

Another quantity of interest is the Green's function:
for $x,y$ with $y-x\in\gplus$ define
\[g(x,y)=\one\{y=x\}+\sum_{m=1}^\infty E_{x}\Bigl[ \exp\Bigl\{-\sum_{k=0}^{m-1} V(T_{X_k}\w,Z_{k+1})\Bigr\} \one_{\{X_m=y \}}  \Bigr].\]

For $\e>0$ let   
    \begin{align}\label{def:R}
    R_\e=\min\{\abs{x}_1:x\in\cG_+(\rangez),\,V(T_x\w,z)\ge\e\text{ for some }z\in\rangez\}.
    \end{align}
Our third main result is the following shape theorem for $g$.

\begin{theorem}\label{shape-green}
Assume the setting of Theorem \ref{shapetheorem}. Assume also that either $P_0$ is transient or there exists an $\e>0$ such that $\E[R^d_\e]<\infty$.  Then for all $\delta>0$ 
	\begin{equation}
	\varlimsup_{ \substack{ \abs{x}_1\to \infty \\ 
			x\in \gplus\cap \cA_\delta} } \dfrac{\abs{ \log g(0,x)+\alpha(x) }}{\abs{x}_1}= 0. \label{lim:shape:green}
	\end{equation}
\end{theorem}

\begin{remark}
If there exists an $\e>0$ such that 
$\P\{V(\w,z)\ge\e\}=1$ for all $z\in\rangez$, then $R_\e=0$, $\P$-almost surely. Also, if $\P$ has a finite range of dependence and $V$ is local and satisfies \eqref{V>0} and $\P\{V(\w,z)=0\ \forall z\in\rangez\}=0$, then there exists an $\e>0$ and $z_0\in\rangez$ such that $\P\{V(\w,z_0)<\e\}<1$. In this case
\[\P(R_\e\ge r)\le\P\{V(T_x\w,z_0)<\e\ \ \forall x\in\cG_+(\rangez):\abs{x}_1<r\}\le\P\{V(\w,z_0)<\e\}^{Cr^d}.\]
Hence $R_\e$ has all moments.
\end{remark}\medskip

Next, we give the proofs of the above claims. In what follows, $C$ will denote a chameleon constant that may change value from term to term. Some technical results are deferred to the Appendix.

 \section{The quenched Lyapunov exponent and the proofs of Theorem \ref{shapetheorem} and Corollary \ref{cor:shape}}\label{sec:Lyap}

The shape theorem for the Lyapunov exponent is proved following a by-now standard route: an application of the subadditive ergodic theorem gives a law of large numbers in rational directions and then the moment assumption on the potential is used along with subadditivity to control the error arising from passing to irrational directions. There are two main technical difficulties. First, we do not assume ergodicity and hence we need to handle the measurability and regularity issues that come with having to define the Lyapunov exponent as a stochastic process. Second, having an arbitrary set of steps makes it nontrivial to construct paths between the various points of the lattice, which is needed when controlling the error coming from passing from rational directions to irrational ones.
\smallskip
 
Recall that we will work with the case of $a$, defined in \eqref{a-def}. The reader can check along the way that, with the appropriate definitions, the analogous results for the case of $a_\infty$, defined \eqref{a-FPP}, go through without much change.\smallskip 
 
 Throughout this section, we will assume that condition \eqref{V>0} holds. Although note that this condition is  vacuous when $0\not\in\Uset$.
 
 We start with the subadditivity and finiteness of $a$. For this, we need the following preliminary lemma.

\begin{lemma}\label{lm:loops}
	Suppose $\sum_{i=1}^kz_i=0$ for some $z_i\in \range.$ Then $z_i\in \rangez$ for all $i.$ Furthermore, if $x=\sum_{i=1}^jz_i$ for some $j\in \lbrace 1,\dots,k-1\rbrace$ and $x=\sum_{\ell=1}^m\tilde{z}_\ell$ or $-x=\sum_{\ell=1}^m\tilde{z}_\ell$ for some $\tilde{z}_\ell \in \range$ then $\tilde{z}_\ell \in\rangez$ for all $\ell.$
\end{lemma}
\begin{proof}
	$\sum_{i=1}^k \frac{1}{k} 	z_i =0\in \Usetz$ implies $z_i\in \Usetz$ for all $i.$ See Theorem 18.1 in \cite{Roc-70}. Next, we write $-x=\sum_{i=j+1}^k z_i$ and hence  $\sum_{\ell=1}^m\tilde{z}_\ell+ \sum_{i=j+1}^k z_i =0.$ By the part we proved already, $\tilde{z}_\ell \in\rangez$ for all $\ell.$ The other case is similar. 
\end{proof}

 \begin{lemma}\label{lemma1}
 	For $\P$-almost every $\w$ and any $x,y,z \in \Z^d $ such that $y-x, z-y\in  \gplus$, 
	\begin{equation} \label{subadd. of a}
	a(\w,x,z)\le a(\w,x,y)+a(\w,y,z). 
	\end{equation}
\end{lemma}

 \begin{proof}
 	If $y\in \lbrace x,z\rbrace,$ then \eqref{subadd. of a} is trivial. Assume $y\not \in \lbrace x,z\rbrace.$ Let  $ \tau_{y,z}=\inf \lbrace {n\ge \tau_y: X_n=z }\rbrace$. Then 
 	\begin{align*}
 	e^{-a(\w,x,y)}\cdot e^{-a(\w,y,z)}&= E_{x}\Bigl[ \exp\Bigl\{-\sum_{k=0}^{\tau_y-1} V(T_{X_k}\w,Z_{k+1})\Bigr\} \one_{\lbrace \tau_y<\infty \rbrace}  \Bigr] \\
 	& \qquad\times  E_{y}\Bigl[ \exp\Bigl\{-\sum_{k=0}^{\tau_z-1} V(T_{X_k}\w,Z_{k+1})\Bigr\} \one_{\lbrace \tau_z< \infty) \rbrace}  \Bigr]\\ 
 	&= E_{x}\Bigl[ \exp\Bigl\{-\sum_{k=0}^{\tau_z-1} V(T_{X_k}\w,Z_{k+1})\Bigr\} \one_{\lbrace \tau_y\le\tau_z< \infty \rbrace}  \Bigr] \\
 	&\qquad +  E_{x}\Bigl[ \exp\Bigl\{-\sum_{k=0}^{\tau_{y,z}-1} V(T_{X_k}\w,Z_{k+1})\Bigr\} \one_{\lbrace\tau_z< \tau_y\le \tau_{y,z} <\infty  \rbrace}  \Bigr]. 
 	\end{align*}
 	If there are no admissible loops from 0 to 0 that go through $y-z,$
 	then $ \one_{\lbrace\tau_z< \tau_y\le  \tau_{y,z} <\infty  \rbrace}=0$, and the term in the last line will be 0.  If such loops exist then Lemma \ref{lm:loops} implies  that any such loop can only take steps in $\rangez$. Thus \eqref{V>0} and $\tau_z<\tau_{y,z}$ imply
 	\[  \exp\Bigl\{-\sum_{k=0}^{\tau_{y,z}-1} V(T_{X_k}\w,Z_{k+1})\Bigr\}\le \exp\Bigl\{-\sum_{k=0}^{\tau_{z}-1} V(T_{X_k}\w,Z_{k+1})\Bigr\}. \] Either way we have
 	\begin{equation*}
 	e^{-a(\w,x,y)}\cdot e^{-a(\w,y,z)}  \le  E_{x}\Bigl[ \exp\Bigl\{-\sum_{k=0}^{\tau_z-1} V(T_{X_k}\w,Z_{k+1})\Bigr\} \one_{\lbrace \tau_z<\infty \rbrace}   \Bigr] = e^{-a(\w,x,z)}.
 	\end{equation*}
 	\eqref{subadd. of a} follows.
 \end{proof}

The next lemma provides an upper bound on $a$. 

\begin{lemma}\label{Lyap:ub}
	$V^+ \in L^1 $ implies $a^+(x,y)$ is in $L^1$  for all $ x,y\in Z^d$ with $y-x\in\gplus$.  
\end{lemma}
\begin{proof}
	Fix an admissible path $x_{0:n}$ from $x$ to $y$ reaching $y$ for the first time at time $n$. Such a path exists because $y-x\in \gplus$. Then
	\begin{align*}
	    \ZLpp(x,y)&=-\log E_x\Bigl[ \exp \Bigl\{ -\sum_{k=0}^{\tau_y-1}V(T_{X_k}\w,Z_{k+1}) \Bigr\}    \one_{\lbrace \tau_y<\infty \rbrace}\Bigr] \\
	    &\le \sum_{k=0}^{n-1}V(T_{x_k}\w,z_{k+1}) -\log P_x(X_{0:n}=x_{0:n}). \qedhere
	\end{align*}
\end{proof} 

To proceed we need the following lemma which is proved in Appendix \ref{Appx2}.
Recall that $\cplus$ is the cone generated by $\range$ and that for $\xi\in\cplus$, $\cC_\xi$ is the unique face of $\cplus$ such that $\xi\in\ri\cC_\xi$. Also, $\range_\xi=\range\cap\cC_\xi$.
 
 \begin{lemma}\label{boundednessofsteps}
 	There exist functions $\gamma_z:\cplus \to \R_+$, $z\in\range$, and a finite positive constant $C$ such that 
 	\[\sum_{z\in\range}\gamma_z(\xi)z=\xi \text{  and  }  \gamma_z(\xi)\le C |\xi|_1\quad\text{for all }\xi\in\cplus.\]
 	If furthermore $\xi \in \gplus$ then $\gamma_{z}(\xi)\in \Z_+$  for all $z\in \range.$
 \end{lemma}
 
Fix $\xi\in\cplus$ and take $\gamma_{z}=\gamma_{z}(\xi)$, $z\in \range$, as in Lemma \ref{boundednessofsteps}.
Then $\gamma_{z}=0$ for $z\not \in \range_{\xi}.$  Let $ \range_\xi'=\lbrace z\in \range_\xi:\gamma_z>0 \rbrace.$  
Define $\kS_\xi$ to be the $\sigma$-algebra of measurable sets $A\in \kS$ such that $T_z^{-1}A=A$ for all  $z\in \range_{\xi}.$ Define 
\begin{equation}\label{def:xhat}
\xhat_t(\xi)=\sum_{z\in\range_\xi}\fl{ t \gamma_z }z=\sum_{z\in\range_\xi'}\fl{ t \gamma_z }z
\end{equation}
and 
\begin{equation}\label{tilalpha-def}
\tilde\alpha(\xi) = \inf_{t>0,t\in\Q} \frac{1}{t} \E[a(0,\xhat_t (\xi))\mid \kS_\xi ].
\end{equation}
If $V^+\in L^1$ then $\E [\tilde\alpha(\xi)] <\infty $, but 
{\it a priori} $\E [\tilde\alpha(\xi)]$ could be $-\infty$.
  Set $\tilde\alpha(0)=0$.

\begin{theorem}\label{thm:atil=lim}Assume $V^+\in L^1$. Fix $\xi \in \cplus$. 
Then $\P$-almost surely and for all $k\in \N$
	\begin{equation}\label{alpha-def}
	\tilde\alpha(\xi)=\lim_{\Q\ni t\to \infty}\frac{\E[a(0,\xhat_t( \xi))\mid \kS_\xi] }{t}= \inf_{n\in \N} \frac{\E[a(0,\xhat_{nk} (\xi))\mid \kS_\xi ]}{nk} .
	\end{equation} 
\end{theorem} 
\begin{proof}
	For a rational $t>0$ let $\overline{a}(t)=\E[a(0,\xhat_t(\xi))\mid \kS_\xi]$.
    Then for rational $t,s>0$,   $\xhat_{t+s}(\xi)-\xhat_t(\xi)-\xhat_s(\xi)\in \gplus$ because  $\fl{(t+s)\gamma_{z} }\ge \fl{t\gamma_{z}}+ \fl{s\gamma_{z}}$  for each $z\in\range$.	By the subadditivity of $a$ we have $\P$-almost surely
	\begin{align*}
	&\overline{a}(t)+\overline{a}(s)+\E[a^+(\xhat_t(\xi)+\xhat_s(\xi), \xhat_{t+s}(\xi))\mid \kS_\xi]\\
	&\qquad\ge  \overline{a}(t)+\overline{a}(s)+\E[a(\xhat_t(\xi)+\xhat_s(\xi), \xhat_{t+s}(\xi))\mid \kS_\xi]\\
	&\qquad= \E[a(0,\xhat_t(\xi))\mid \kS_\xi]+\E[a(0,\xhat_s(\xi))\mid \kS_\xi]  
	+\E[a(\xhat_t(\xi)+\xhat_s(\xi), \xhat_{t+s}(\xi))\mid \kS_\xi]\\
	&\qquad= \E[a(0,\xhat_t(\xi))\mid \kS_\xi]+\E[a(\xhat_t(\xi),\xhat_t(\xi)+\xhat_s(\xi))\mid \kS_\xi]\\
	&\qquad\qquad\qquad+\E[a(\xhat_t(\xi)+\xhat_s(\xi), \xhat_{t+s}(\xi))\mid \kS_\xi]\\
	&\qquad= \E[a(0,\xhat_t(\xi))+a(\xhat_t(\xi),\xhat_t(\xi)+\xhat_s(\xi))+a(\xhat_t(\xi)+\xhat_s(\xi), \xhat_{t+s}(\xi))\mid \kS_\xi]\\
	& \qquad\ge\E[a(0,\xhat_{t+s}(\xi))\mid \kS_\xi] 
	=\overline{a}(t+s).
	\end{align*}
	The second equality comes because for any $A\in \kS_\xi$ we have $T_{\xhat_{t}(\xi)}^{-1}A=A. $
	Next, note that
	$\fl{ (t+s)\gamma_{z}}- \fl{t\gamma_{z}}- \fl{s\gamma_{z} }\le 1.$
	Hence, $\abs{\xhat_{t+s}(\xi)-\xhat_t(\xi)-\xhat_s(\xi) }_1\le   \sum_{z\in\range_\xi'}\abs{z}_1= c_1.$ 
	Also, \[c_2(\w)=\max_{x\in \gplus,\abs{x}_1\le c_1}\E[a^+(0,x)\mid \kS_\xi]\in [0,\infty)\] because $V^+\in L^1.$  It follows that $c_2+\overline{a}(t)+\overline{a}(s)\ge \overline{a}(t+s).$ 
	Note that for each $t>0$,  $\overline{a}$ is bounded on $[0,t]\cap\Q$ because $\E[a^+(0,x)\mid \kS_\xi]<\infty$ and there are finitely many $x \in\lbrace \xhat_s(\xi): s\le t \rbrace$. Fekete's Lemma 
    now gives
	\begin{equation*}
	\frac{\overline{a}(t)}{t}\xrightarrow[\Q\ni t \to \infty] {}\inf_{n\in \N}\frac{\overline{a}(nk)}{nk}= \inf_{\Q\ni s>0}\frac{\overline{a}(s)}{s}= \tilde\alpha(\xi)
	\end{equation*}
	which proves  \eqref{alpha-def}. 
\end{proof}

For now, $\tilde\alpha(\xi)$ is a random variable, defined up to a null set that may depend on $\xi$ and on the particular choice of $(\gamma_{z})_{z\in\range }$. We next show that the limit in \eqref{alpha-def} holds more generally and does not depend on the specific choice of the coefficients $(\gamma_{z})_{z\in\range }$. First we handle the case of $\xi \in \gplus.$ For $x\in \gplus\setminus\lbrace0\rbrace$ let $\cI_x$ be the $\sigma$-algebra generated by $A\in \kS$ such that $T_x^{-1}A=A.$ Set $\overline\alpha(0)=0$ and for $x\in\gplus\setminus\{0\}$ let 
\[ \overline{ \alpha}(x)=\inf_{n\ge 1} \dfrac{\E[a(0,nx)\mid \cI_x]}{n}.\]

\begin{theorem}\label{a.s. theorem} Assume $V^+\in L^1.$
	Fix $x \in \gplus.$ 	Then 
	\begin{equation}\label{a=lim}
	\lim_{n\to \infty}\frac{a(0,nx)}{n}=\overline{\alpha}(x) \quad \P \text{-almost surely.}  
	\end{equation}
	The limit also holds  in $L^1$ if  $\E[\overline\alpha(x)]>-\infty$ and in this case $\overline\alpha(x)=\tilde\alpha(x)$ 
	$\P$-almost surely.
\end{theorem}

\begin{proof}
	Fix $x\in\gplus$ and for nonnegative integers $m\le n$ let $X_{m,n}=a(mx, nx)$. The subadditivity of $a$, the invariance of $\P$ under the action of the shift $T_x$, and the fact that $a^+(0,x)\in L^1$ ensure that the assumptions of Liggett's subadditive ergodic theorem in \cite{Lig-85} are satisfied. 
	Thus 
	\[\lim_{n\to \infty}\frac{a(0,nx)}{n}= \inf_{n\in\N} \frac{\E[a(0,nx)\mid \cI_x]}{n}=\overline{ \alpha}(x) \quad \P\text{-almost surely.}\]
	The same theorem says the limit also holds in $L^1$ if $\E[\overline{ \alpha}(x)]>-\infty$.
Since $\xhat_n(x)=nx$ if $x \in \gplus$, this and \eqref{alpha-def} imply that $\overline{\alpha}(x)=\tilde\alpha(x)$, $\P$-almost surely.
\end{proof}

Next, we handle the case of $\xi\in \cplus'$, where $\cplus'$ is the rational cone generated by $\range$.

\begin{lemma}\label{lm:atil-gen}
	Assume $V^+\in L^1$. Fix a face $\cA$ of $\cplus$ and assume $\E[\overline\alpha(x)]>-\infty$ for all $x\in\gplus\cap\cA$.
	The following holds $\P$-almost surely: For any $(\gamma_{z}')_{z\in \range\cap\cA}\in \Q^{\range\cap\cA}_+$ 
	\begin{equation}\label{lima=atil}
	\lim_{t\to \infty}\dfrac{a\bigl(0, \sum_{{z}\in \range\cap\cA}\fl{t\gamma_{z}'}z \bigr)}{t}=\tilde \alpha\Bigl(\sum_{z\in \range\cap\cA}\gamma_{z}'z\Bigr).
	\end{equation}
\end{lemma}

\begin{proof}
Define $\tilde\alpha(\xi)$ for all $\xi\in\cplus'\cap\cA$ via \eqref{tilalpha-def}, using the representation  
$(\gamma_{z}(\xi))_{z\in \range}$ from Lemma \ref{boundednessofsteps}.	
Applying Theorem \ref{a.s. theorem}, let $\Omega_0$ be the full $\P$-measure event on which the limit in \eqref{a=lim} holds and $\overline\alpha(x)=\tilde\alpha(x)$, for all $x\in\gplus\cap\cA\setminus\{0\}$.
Fix an integer $L\ge 1$ and take a representation $(\gamma_{z}')_{z\in \range\cap\cA}$ as in the claim but with $\max_{ z\in \range\cap\cA}\gamma_{z}'\le L$. Let $\xi=\sum_{z\in\range\cap\cA}\gamma_z' z$.

Abbreviate  $\xhat_{t}'(\xi)=\sum_{{z}\in \range}\fl{t\gamma_{z}'}z.$
Take $k\in\N$ such that $k\gamma_{z}'\in \Z_+$ for all $z\in\range.$ Then $\xhat_{nk}'(\xi)=nk\xi$ for all $n\in \N.$ 
By limit \eqref{a=lim} with $x=k\xi\in\gplus\cap\cA$,
	\[ \lim_{n\to \infty} \dfrac{a(0,\xhat_{nk}'(\xi))}{nk} =\lim_{n\to \infty} \dfrac{a(0,nk\xi)}{nk} =\tilde \alpha(\xi).  \]
	For $t\ge k$ let $n\in\N$ be such that $nk\le t<(n+1)k.$ Then for $z\in \range\cap\cA$ with $\gamma_{z}'>0$ we have $k\gamma_{z}'\ge 1$ and hence
    \[ (n-1)k\gamma_{z}'\le nk\gamma_{z}'-1\le t\gamma_{z}'-1\le \fl{t\gamma_{z}'} \le t\gamma_{z}'\le (n+1)k\gamma_{z}'. \]
	When $\gamma_{z}'=0$ we still have 
	\[ (n-1)k\gamma_{z}'\le \fl{t\gamma_{z}'} \le (n+1)k\gamma_{z}'. \]
	Thus, $\xhat_{t}'(\xi)$ is accessible from $(n-1)k\xi$ by an admissible path and $(n+1)k\xi$ is accessible from $\xhat_{t}'(\xi)$ by an admissible path. The endpoints in both cases are at most $2k\sum_{{z}\in \range\cap\cA}\gamma_{z}'\abs{z}_1\le 2k L\abs{\range}\max_{ z\in \range }\abs{z}_1=C$ away from each other.
	
	By subadditivity, if we set 
		\[A=\max\{a^+(0,y)\vee a^+(-y,0):y\in\gplus,\abs{y}\le C\}\in L^1,\] 
	then 
	\begin{align*}
	a(0,(n+1)k\xi)-A\circ T_{(n+1)k\xi}\le & a(0,(n+1)k\xi)-a(\xhat_{t}'(\xi),(n+1)k\xi) \\ 
	\le &a(0,\xhat_{t}'(\xi))\\
	\le & a(0,(n-1)k\xi)+a((n-1)k\xi, \xhat_{t}'(\xi))\\
	\le &a(0,(n-1)k\xi)+A\circ T_{(n-1)k\xi}.
	\end{align*}
By stationarity, $A\circ T_{n\ell\xi}/n\to0$ almost surely.  
Let $\Omega_L$ be the full measure event on which $A\circ T_{n\ell \xi}/n\to 0$ for any $\ell\in \N$ such that $\ell \xi\in \gplus.$ Divide the above by $t$ and take $t\to \infty$ to get \eqref{lima=atil} for any rational representation 
	$(\gamma_{z}')_{z\in \range\cap\cA}$ with $\gamma_{z}'\le L$ for all $z\in \range\cap\cA.$
	The claim of the lemma holds on $\cap_{L\in \Z_+}\Omega_L.$ \qedhere
\end{proof}

Now that we know that the limit in $\eqref{lima=atil}$ is independent of the choice of the rational representation, we can prove some some basic properties of $\tilde\alpha$ when restricted to rational arguments.
\begin{theorem}\label{1st theorem}
	Assume $V^+\in L^1$. Fix a face $\cA$ of $\cplus$ and assume $\E[\overline\alpha(x)]>-\infty$ for all $x\in\gplus\cap\cA$.
	There exists a constant $C<\infty$ {\rm(}only depending on $\range${\rm)}
	such that for
	$\P$-almost every $\w$, for all  $s\in \Q_+$, and for all $\xi,  \zeta \in \cplus'\cap\cA$, 
	\begin{align}
	& \tilde\alpha(s\xi)=s\tilde\alpha(\xi), \quad  \tilde\alpha(\xi)+\tilde\alpha(\zeta)\ge \tilde\alpha(\xi +\zeta), \label{atil-homo} \\[1ex]  
	&  \text{and}\quad \tilde\alpha(\xi)\le  C \max_{z\in \range_\xi}\bigl( \E \bigl[V^+(\w,z)\mid \kS_\xi\bigr]-\log p(z)  \bigr) \abs{\xi}_1 . \label{atil-bded}
	\end{align}   
\end{theorem}

\begin{proof}
	If  $\xi=\sum_{z\in \range\cap\cA}\gamma_z z $ with rational coefficients $\gamma_z$,
	then $s\xi=\sum_{z\in \range\cap\cA}s\gamma_z z$, with rational coefficients $s\gamma_z$. 
	Applying Lemma \ref{lm:atil-gen} twice gives
	\begin{align*}
	\tilde \alpha(s\xi)=&\lim_{t\to \infty}\dfrac{ a(0,\sum_{z\in \range} \fl{ts \gamma_z} z)  }{t}
	=s \lim_{t\to \infty}\dfrac{a(0,\sum_{z\in \range} \fl{t\gamma_z} z) }{t}
	=s\tilde \alpha(\xi).
	\end{align*}
	This proves the homogeneity.  
	Next, take $\xi,\zeta\in\cplus'\cap\cA$ and let $m\in\N$ be such that $m\xi$ and $m\zeta$ are in $\gplus\cap\cA$. Then
	\begin{align*}
	\E[a(0,nm\xi)\mid \kS_\xi]+\E[a(0,nm\zeta)\mid\kS_\xi]=& \E[a(0,nm\xi)+ a(nm\xi,nm(\xi+\zeta)\mid \kS_\xi] \\
	\ge& \E[a(0,nm(\xi+\zeta))\mid \kS_\xi].
	\end{align*}
	The equality comes because for $A\in \kS_\xi,$ $T_{nm\xi}^{-1}A=A.$
	Divide by $nm$, take $n\to \infty$, and use either Theorem \ref{thm:atil=lim} or Theorem \ref{a.s. theorem} to get 
	\[\tilde\alpha(\xi)+\tilde\alpha(\zeta)\ge \tilde\alpha(\xi+\zeta).\]
	For  \eqref{atil-bded} write $\xi$ using the coefficients $\gamma_z=\gamma_z(\xi)$ given by Lemma \ref{boundednessofsteps} and recall that then $\gamma_{z}\le C\abs{\xi}_1$ for all $z\in\range$. 
	Let $n=\sum_z\fl{t\gamma_z}$ and pick  any admissible path $x_{0:n}$ that takes $\fl{t \gamma_z }$ $z$-steps for $z\in \range$, to go from $0$ to $\xhat_t(\xi)$. Then
	\begin{align*}
	\E\bigl[ a(0,\xhat_t(\xi))\mid \kS_\xi\bigr] 
	&\le\E\Bigl[\sum_{i=0}^{n-1}V^+(T_{x_i}\w,z_{i+1})\,\Big|\,\kS_\xi\Bigr]-\log P_0(X_{0:n}=x_{0:n})\\
	&=\sum_{i=0}^{n-1}\bigl(\E[V^+(\w,z_{i+1})\mid\kS_\xi]-\log p(z_{i+1})\bigr)\\
	&\le \max_{z\in \range_\xi} \bigl( \E[ V^+(\w,z)\mid \kS_\xi]-\log p(z) \bigr) \cdot  \sum_{z\in \range }\fl{t\gamma_z }.
	\end{align*}
	But 
    \[
	\frac{1}{t}\sum_{z\in \range}\fl{t\gamma_z }\mathop{\longrightarrow}_{t\to\infty} \sum_{z\in \range}\gamma_z \le C\abs{\range}\cdot\abs{\xi}_1.\]
	
	Bound \eqref{atil-bded} follows  and the theorem is proved. 
\end{proof}

We now can define the limiting quenched Lyapunov exponent $\alpha$ out of the function $\tilde\alpha$. 

\begin{theorem}\label{alphaext}
	 Fix a face $\cA\subset\cplus.$
	Assume $V^+\in L^1$ and $\E[\overline\alpha(x)]>-\infty$ for all $x\in\gplus\cap\cA$.
	Then $\P$-almost surely,  there exists a unique finite locally Lipschitz convex homogeneous function $\alpha$ on $\ri\cA$ such that $\tilde \alpha=\alpha$ on $\cplus'\cap \ri\cA$. 
\end{theorem}

\begin{proof} 
The homogeneity and subadditivity imply that $\tilde \alpha$ is convex on $\cplus'\cap \cA$. 
Since $\E[\overline\alpha(x)]>-\infty$ for $x\in\gplus\cap\cA$, Theorem \ref{a.s. theorem} implies that $\P$-almost surely $\tilde\alpha(x)=\overline\alpha(x)$ for all $x\in\gplus\cap\cA$ and hence 
$\tilde\alpha(x)>-\infty$ for all $x\in\gplus\cap\cA$.
Homogeneity implies then that 
$\tilde\alpha(\xi)>-\infty$ for all $\xi\in\cplus'\cap\cA$.
By putting this together with the properties established in Theorem \ref{1st theorem} we will now prove that $\tilde\alpha$ is locally bounded below on $\cplus'\cap\ri\cA$. To this end, take $\e>0$ and $\zeta\in\cplus'\cap\ri\cA$. We will give a lower bound on $\tilde\alpha(\xi)$, uniformly in $\xi\in\cplus'\cap\ri\cA$ with $\abs{\xi-\zeta}_1<\e$.

Take an integer $k>\e^{-1}\abs{\range}\cdot\max_{z\in\range}\abs{z}_1$ and
let $\eta=\zeta+k^{-1}\sum_{z\in\range\cap\cA}z\in\cplus'\cap\ri\cA$.
We can write $\eta=\sum_{z\in\range\cap\cA}\bar\gamma_z z$ with rational $\bar\gamma_z\ge k^{-1}$ 
for all $z\in\range\cap\cA$.
Take a rational $0<t<1$ such that
		\[t\le\frac1{C(\abs{\eta}_1+2\e)k},\]
where $C$ is the constant in Lemma \ref{boundednessofsteps}.
Now consider $\xi\in\cplus'\cap\ri\cA$ with $\abs{\xi-\zeta}_1<\e$. Then $\abs{\xi-\eta}_1<2\e$.
Take any integer $m\in\N$ such that $m\xi\in\gplus$ and let $\gamma_z(m\xi)\in\Z_+$, $z\in\range\cap\cA$, 
be the coefficients given by Lemma \ref{boundednessofsteps}. The choice of $t$ implies that 
	\[m\bar\gamma_z\ge m/k\ge Ctm\abs{\xi}_1\ge t\gamma_z(m\xi)\]
for all $z\in\range\cap\cA$ and hence $m\eta-tm\xi\in\cplus'\cap\cA$. 
	 By the inequalities in \eqref{atil-homo} and \eqref{atil-bded} we have
		\begin{align*}
		\tilde\alpha(m\eta)
		&\le  \tilde\alpha(tm\xi)+\tilde\alpha(m\eta-tm\xi)\le \tilde\alpha(tm\xi)+C(\w)m\abs{\eta-t\xi}_1\\
		&\le \tilde\alpha(tm\xi)+C(\w)m\bigl((1+t)\abs{\eta}_1+2t\e\bigr).\end{align*}
	Note that in the first inequality above, in the application of \eqref{atil-bded}, there is a dependence on $m\eta-tm\xi$ through the conditional expectation given $\kS_{m\eta-tm\xi}$ and so it may appear that the constant $C(\w)$ may not be uniform. This presents no issue, however, as there are only finitely many sigma algebras which can appear in the conditional expectation.  
	Using the homogeneity in \eqref{atil-homo} and rearranging one gets
		\[\tilde\alpha(\xi)\ge t^{-1}\tilde\alpha(\eta)-C(\w)\bigl((t^{-1}+1)\abs{\eta}_1+2\e\bigr).\]
This proves that $\tilde\alpha$ is locally bounded below on $\cplus'\cap\ri\cA$.  
Bound \eqref{atil-bded} implies $\tilde\alpha$ is also locally bounded above. 

Now that we have shown the local boundedness of $\tilde\alpha$ we will show that it is Lipschitz on any small enough ball in $\cplus'\cap \ri \cA$.

Take $\xi_0\in\cplus'\cap \ri \cA$. Take $\e>0$ such that if $\xi$ is in the linear span of $\cA$ and $\abs{\xi-\xi_0}_1<3\e$ then $\xi\in\ri\cA$. 
Take $\xi,\zeta\in\cplus'\cap \ri \cA$ with $\abs{\xi-\xi_0}_1<\e$ and $\abs{\zeta-\xi_0}_1<\e$. 
Note that $\abs{\xi+t^{-1}(\zeta-\xi)-\xi_0}_1$ is continuous in $t$, converges to $\abs{\zeta-\xi_0}_1<\e$ as $t$ increases to $1$ and 
converges to $\infty$ as $t$ decreases to $0$. Hence, one can pick a rational $t\in(0,1)$ such that
	\[2\e<\abs{\xi+t^{-1}(\zeta-\xi)-\xi_0}_1<3\e.\]
In particular, $\xi+t^{-1}(\zeta-\xi)\in\ri\cA$. Lemma \ref{in-cone} says then this is also in $\cplus'$.
Furthermore, the fact that
	\[\abs{\xi+t^{-1}(\zeta-\xi)-\xi_0}_1<\e+t^{-1}\abs{\zeta-\xi}_1\] 
and the first inequality in the above display imply that $t<\e^{-1}\abs{\zeta-\xi}_1$.
Now write
	\[\tilde\alpha(\zeta)=\tilde\alpha\Bigl(t\bigl(\xi+t^{-1}(\zeta-\xi)\bigr)+(1-t)\xi\Bigr)\le
	t\tilde \alpha(\xi+t^{-1}(\zeta-\xi))+(1-t)\tilde\alpha(\xi),\]
from which follows 
	\[\tilde\alpha(\zeta)-\tilde\alpha(\xi)\le\e^{-1}C(\xi_0,\e)\abs{\zeta-\xi}_1\]
with $C(\xi_0,\e)=2\sup\{\abs{\tilde\alpha(\eta)}:\abs{\eta-\xi_0}_1<3\e\}$. The other bound comes by switching the roles of $\xi$ and $\zeta$.

By a standard finite subcover argument, the above Lipschitz continuity shows that if $K \subset \ri \cA$ is compact, then $\tilde{\alpha}$ is uniformly continuous on $K \cap \cplus'$. This allows us to extend $\tilde\alpha$ uniquely to a continuous function on $\ri \cA$ and then \eqref{atil-homo} and \eqref{atil-bded}  and consequently convexity also hold for $\alpha$.
\end{proof}

Now that the process $\alpha$ has been defined, we prove a stronger version of Theorem \ref{shapetheorem}.
Recall the definition of $\cA_\delta$ from \eqref{cAdel}.

\begin{theorem}\label{shapetheorem.aux}
	Assume $V^+\in L^1$.
	Fix a face $\cA\not=\{0\}$ of $\cplus$ {\rm(}possibly $\cplus$ itself\,{\rm)}. Let $\rangegen=\range\cap\cA$.
	 Assume $V^+(\w,z)\in \cL_{z,\rangegen}$ for each $z\in\rangegen\setminus\{0\}$ and $V^+(\w,0)\in\cL_{\hat z,\rangegen}$ for some $\hat z\in\rangegen\setminus\{0\}$.
	 Assume also that $\E[\overline\alpha(x)]>-\infty$ for all $x\in\gplus$.
	Then 
	$\P$-almost surely	
	and for all $ \delta>0$ 
	\begin{align}\label{lim:shape.aux}
	\varlimsup_{ \substack{ \abs{x}_1\to \infty \\ 
			x\in \gplus\cap \cA_\delta} } \dfrac{\abs{ a(0,x)-\alpha(x) }}{\abs{x}_1}= 0.
	\end{align}
	If furthermore $\P$ is ergodic under $\{T_z:z\in\range\cap\cA\}$, then $\alpha$ is deterministic on $\ri\cA$.
\end{theorem}

 \begin{proof}
The proof comes by  way of contradiction.  Assume that with positive probability  there exists an $\e >0$ and  a sequence $x_{\ell} \in \gplus \cap \cA_\delta $ such that $\abs{x_\ell}_1\to \infty$ and 
\[\dfrac{\abs{ a(0,x_\ell)-\alpha(x_\ell) }}{|x_\ell|_1}\ge \e.\]

Let $\Omega_0'$ be the intersection of the event in the previous paragraph with the full-measure events on which Lemma \ref{lemma1} holds and the limit \eqref{def:cL} is satisfied with $g(T_{x+kz}\w)=V(T_{x+kz}\w,z)$, for each  $ z\in\rangegen\setminus\{0\}$, and with $g(T_{x+k\hat z}\w)=V(T_{x+k\hat z}\w,0)$.

 Apply Lemma \ref{boundednessofsteps} to write $x_\ell=\sum_{z\in \rangegen}b_{\ell,z} z$ with $ b_{\ell,z}\in \Z_+$ such that $ b_{\ell,z} \le C |x_\ell|_1$ for all $z\in \rangegen$. By compactness, we can find a subsequence $\ell_n$ and $\gamma_z\in [0,C]$ such that $b_{\ell_n,z}/\abs{x_{\ell_n}}_1\to \gamma_z$ for all $z\in \rangegen.$ Then $x_{\ell_n}/\abs{x_{\ell_n}}_1\to \xi=\sum_{z\in \rangegen}\gamma_z z$.
 Abbreviate $\ell_n$ by writing just $n.$ 
 Since $x_n/|x_n|_1\in\cA_\delta$ we have $\alpha(x_n)/|x_n|_1= \alpha\left( x_n/|x_n|_1\right)\to \alpha(\xi).$ Therefore, for $n$ large enough
 \begin{equation}\label{contr}
\Bigl|  \frac{a(0,x_n )}{|x_n |_1}-\alpha(\xi)  \Bigr| \ge \e/2.
 \end{equation}

Fix $\e_1\in(0,1)$.  For $m\in\N$
let $\overline{r}_m=\sum_{{z}\in \rangegen} \ce{m(\gamma_z+\e_1)}$ and $\overline{\kappa}_{m,z}=\ce{m(\gamma_{z}+\e_1)}/\overline{r}_m$, for $z\in\rangegen$.   Let $\rho=\sum_{z\in\rangegen}\gamma_z\le C\abs{\rangegen}$. Then 
$\overline{r}_m/m\to \rho+\e_1\abs{\rangegen}$  and 
 \[\overline{\kappa}_{m,z}\to \frac{\gamma_{z}+\e_1}{ \rho+\e_1\abs{\rangegen}}\quad \text{as}\quad m\to \infty,\]
 There exists an $m_0$ such that for any integer $m\ge m_0$ and any $z\in \rangegen$
 \begin{equation}\label{kappa-bd}
 \frac{\gamma_{z}+\e_1/2}{ \rho+\e_1\abs{\rangegen}}\le \overline{\kappa}_{m,z} \le \frac{\gamma_{z}+2\e_1}{ \rho+\e_1\abs{\rangegen}}\,.
 \end{equation}
 Take $m\ge m_0$. Let   
 \begin{equation*}
 \overline{ \zeta}_m=\sum_{{z}\in \rangegen}\overline{\kappa}_{m,z} z, \quad \overline{k}_n = \left\lfloor \frac{( \rho+\e_1\abs{\rangegen})|x_n |_1}{\overline{r}_m}\right\rfloor, \quad \text{and}\quad \overline{s}_z^{(n)}=\overline{r}_m	\overline{k}_n\overline{\kappa}_{m,z}-b_{n,z}.
 \end{equation*}
 Then  for any $z\in\rangegen$
 \begin{equation}\label{s/x:lim}
 \frac{\overline{ s}_z^{(n)}}{|x_n |_1}\to(\rho+\e_1\abs{\rangegen})\overline{\kappa}_{m,z}-\gamma_{z}\ge \e_1/2>0 \quad \text{as } n\to \infty.
 \end{equation}
 
 Thus, $\overline{s}_z^{(n)}\ge 0$ for large enough $n$
 and then $\overline{r}_m \overline{k}_n \overline{ \zeta}_m-  x_n= \sum_{z\in \rangegen} (\overline{r}_m \overline{k}_n  \overline{\kappa}_{m,z} -b_{n,z}) z \in \gplus.$ 
 By the subadditivity  of $a$, 
 \begin{equation}\label{eq3}
 a(0, \overline{r}_m\overline{k}_n\overline{ \zeta}_m)-a(x_n, \overline{r}_m\overline{k}_n\overline{ \zeta}_m)\le a(0,x_n)
 \end{equation}
 Similarly, let $\underline{r}_m=\sum_{{z}\in \rangegen}\fl{m\gamma_{z}}$ and $\underline{\kappa}_{m,z}=\fl{m\gamma_{z}}/\underline{r}_m\le1$ for all $z\in\rangegen$.  We have $\underline{r}_m/m\to \rho$ and $\underline{\kappa}_{m,z}\to \gamma_z/\rho$ as $m\to \infty$. Let $\rangegengen=\lbrace z\in \rangegen:\gamma_z>0\rbrace$ and $\delta'= \min_{ z\in \rangegengen }\gamma_{z}/\rho >  0$.  There exists an $m_1\ge m_0$ such that for any integer $m\ge m_1$ and any $z\in\rangegengen$ we have $\underline{\kappa}_{m,z}\in [\delta'/2,1]$ 
 and $\abs{\rho\underline{\kappa}_{m,z}-\gamma_{z}}<\e_1$. Fix $m\ge m_1$.
 
 Now, suppose $\e_1<\delta'\rho/4$ and let
 \begin{equation*}
 \underline{\zeta}_m=\sum_{{z}\in \rangegen}\underline{\kappa}_{m,z} z,\quad\underline{k}_n =  \left\lfloor \frac{(\rho-4\e_1/\delta')\abs{x_n}_1}{\underline{r}_m}\right\rfloor\quad \text{and}\quad \underline{s}_z^{(n)}=	b_{n,z}- \underline{r}_m\underline{k}_n\underline{\kappa}_{m,z} .
 \end{equation*} 
 Then for $z\in\rangegengen$ we have as $n\to\infty$
 \begin{align*}
 \frac{\underline{s}_z^{(n)}}{|x_n|_1}&\to\gamma_{z}- (\rho-4\e_1/\delta')\underline{\kappa}_{m,z}
 \ge \gamma_{z}-\rho\underline{\kappa}_{m,z}+2\e_1\ge\e_1>0.
 \end{align*}
When $z\in\rangegen\setminus\rangegengen$, $\gamma_z=\underline\kappa_{m,z}=0$. 
 Thus, for $n$ large,  $x_n-\underline{r}_m\underline{k}_n \underline{\zeta}_m = \sum_{z\in \rangegen} (b_{n,z}- \underline{r}_m\underline{k}_n\underline{\kappa}_{m,z} )z \in \gplus$.    
By subadditivity
 \begin{equation}\label{eq4}
 a(0,x_n)\le a(0,  \underline{r}_m\underline{k}_n \underline{\zeta}_m)+a(\underline{r}_m\underline{k}_n \underline{\zeta}_m,x_n).
 \end{equation}
 Note also that  
 \begin{equation}\label{rk:lim}
\frac{ \overline{r}_m \overline{k}_n}{\abs{x_n}_1}\to 	 \rho+|\rangegen|\e_1 \quad 
 \text{and}\quad  
 \frac{\underline{r}_m\underline{k}_n }{\abs{x_n}_1}\to \rho-4\e_1/\delta' \quad \text{as }  n\to \infty.
 \end{equation}
In particular, we have for $n$ large
 \[
\underline{r}_m\underline{k}_n\ge|x_n|_1\left(\rho-4\e_1/\delta'-\e_1\right).\]
 
 Next, observe that if  $z\in \rangegengen$ and $\e_1\in(0,1)$ is small enough to have $4\e_1/\delta'+\e_1<\rho$,
 then both  $\overline{r}_m \overline{k}_n   \overline{\kappa}_{m,z} -b_{n,z}$ and $b_{n,z}-\underline{r}_m \underline{k}_n  \underline{\kappa}_{m,z} $ 
 are bounded above by 
 \begin{align*}
 \overline{r}_m \overline{k}_n  \overline{\kappa}_{m,z}- \underline{r}_m \underline{k}_n    \underline{\kappa}_{m,z}
 &\le  \abs{x_n}_1(\rho+\abs{\rangegen}\e_1)\overline{\kappa}_{m,z}- \abs{x_n}_1(\rho-4\e_1/\delta'-\e_1)\underline{\kappa}_{m,z}  \\ 
 &\le  \abs{x_n}_1(\gamma_{z}+2\e_1)- \abs{x_n}_1(\rho-4\e_1/\delta'-\e_1)(\gamma_{z}-\e_1)/\rho  \\ 
 &=\abs{x_n}_1\e_1\Bigl(3+\frac{4(\gamma_{z}-\e_1)}{\delta'\rho}+\frac{\gamma_{z}-\e_1}\rho\Bigr)  \\
 &\le 4(1+1/\delta')\abs{x_n}_1\e_1=c_1\abs{x_n}_1\e_1.  
 \end{align*}
 On the other hand, for $z\in \rangegen \setminus\rangegengen$, $\gamma_{z}=\underline{\kappa}_{m,z}=0$ and 
 \begin{align*}
 0\le 	\overline{r}_m \overline{k}_n  \overline{\kappa}_{m,z}- \underline{r}_m \underline{k}_n    \underline{\kappa}_{m,z}
 &\le 2\abs{x_n}_1\e_1\le 4(1+1/\delta')\abs{x_n}_1\e_1=c_1\abs{x_n}_1\e_1.
 \end{align*}
 
We next develop an upper bound for $a(x_n, \overline{r}_m\overline{k}_n\overline{ \zeta}_m ).$  
Fix a path from $x_n$ to $\overline{r}_m\overline{k}_n\overline{ \zeta}_m$	that takes $\overline{ s}_z^{(n)}$ $z$-steps for each $z\in \rangegen$.
Recall that $\hat{z}\in \rangegen\setminus\{0\}$ and that
  $b_{n,0}/\abs{x_n}_1\to\gamma_0$. This, \eqref{kappa-bd}, and \eqref{s/x:lim} imply that for large $n$
 \begin{align*}
 \frac{\overline{s}_0^{(n)}}{\overline{ s}_{\hat{z}}^{(n)}}
 \le \frac{\bigl((\rho+\abs{\rangegen}\e_1)\overline{\kappa}_{m,0}-\gamma_{0}+\e_1\bigr)\abs{x_n}_1}{|x_n|_1\e_1/4}     
 \le \frac{2\e_1+\e_1}{\e_1/4}=12.   
 \end{align*}
 This tells us the ratio of zero steps to $\hat{z}$ steps is at most $12$. 
Rearrange the path as follows. 
Start the path  with blocks of a $\hat{z}$ steps followed by at most $12$ zero steps, until the $\hat{z}$-steps and zero steps have been exhausted. 
Next, fix an ordering of  $\range\setminus\lbrace0,\hat{z}\rbrace=\lbrace z_1, z_2,\ldots\rbrace$ 
and arrange the rest of 
the path  to take first all its $z_1$ steps then all the $z_2$ steps and so on. Also note that any point $y$ on the path is such that $y\in \gplus(\rangegen)$ and 
 \[ \abs{y}_1\le \abs{x_n}_1+\Bigl(\overline{ r}_m\overline{ k}_n-\sum_{z\in\rangegen}b_{n,z}\Bigr)\max_{ z\in \rangegen }\abs{z}_1 \le \abs{x_n}_1\Bigl(1+\abs{\rangegen}(C+\e_1)\max_{ z\in \rangegen }\abs{z}_1\Bigr)=c_2\abs{x_n}_1. \]
 Thus  
 \begin{align*}
 a(x_n,\overline{r}_m\overline{k}_n\overline{ \zeta}_m)
 &\le 
 \abs{\rangegen}\max_{ \substack{y\in\gplus(\rangegen) \\ \abs{y}_1\le c_2 \abs{x_n}_1 }}\max_{z\in \rangegen\setminus \lbrace 0\rbrace} \sum_{0\le i\le c_1\e_1 \abs{x_n}_1} V^+(T_{y+iz}\w,z)\\
 &\qquad+12 \max_{ \substack{y\in\gplus(\rangegen) \\ \abs{y}_1\le c_2 \abs{x_n}_1 }}\sum_{0\le i\le c_1\e_1 \abs{x_n}_1} V^+(T_{y+i\hat z}\w,0)-c_1\e_1 |x_n|_1\min_{ z\in \rangegen}\log p(z).
 \end{align*}
Divide through by $\abs{x_n}_1$ and take $n \to \infty$ to obtain
 \begin{align*}
 \varlimsup_{n \to \infty} \dfrac{a(x_n,\overline{r}_m\overline{k}_n\overline{ \zeta}_m)}{|x_n|_1} 
 &\le 
 \abs{\range'}\varlimsup_{n \to \infty} \max_{ \substack{y\in\gplus(\rangegen) \\ \abs{y}\le c_2 |x_n|_1 }}\max_{z\in \rangegen\setminus \lbrace 0\rbrace}\frac{1}{|x_n|_1} \sum_{0\le i\le c_1\e_1 \abs{x_n}_1} V^+(T_{y+iz}\w,z)\\
&\qquad+12 
\varlimsup_{n \to \infty} \max_{ \substack{y\in\gplus(\rangegen) \\ \abs{y}\le c_2 |x_n|_1 }}\frac{1}{|x_n|_1} \sum_{0\le i\le c_1\e_1 \abs{x_n}_1} V^+(T_{y+i\hat z}\w,0)
 -c_1\e_1 \min_{ z\in \rangegen}\log p(z). 
 \end{align*}
 Fix any $\e_2>0$. 
 Since $\w\in\Omega_0'$ we can find $\e_1$ small enough so that the right-hand side in the above display  is smaller than $\e_2$. 
 Similarly, 
 \[
 \varlimsup_{n\to \infty} \dfrac{ a(\underline{r}_m\underline{k}_n\underline{\zeta}_m,x_n) }{|x_n|_1}\le \e_2. \] 
 With equations
 \eqref{eq3},
 \eqref{eq4}, and \eqref{rk:lim}  we conclude that 
 \begin{equation*}
 -\e_2+  \left(\rho+|\rangegen|\e_1 \right)  \alpha\left(\overline{\zeta}_m\right) \le \varliminf_{n\to\infty}  \frac{  a(0,x_{n} )}{|x_n|_1}\le \varlimsup_{n\to\infty}  \frac{ a(0,x_n)}{|x_n|_1}\le  \left(\rho-4\e_1/\delta' \right)  \alpha \left(\underline{\zeta}_m \right) +\e_2. 
 \end{equation*}
 Since  $\alpha$ is continuous on $\ri \cA$,  $\xi\in\ri \cA$, and $\underline{\zeta}_m$ and $\overline{\zeta}_m$ are both in $\cA$,
we have for $\e_1>0$ small enough
 \[ \alpha\bigl(\underline{\zeta}_m\bigr)\to \alpha(\xi) \quad \text{and} \quad \alpha\bigl(\overline{\zeta}_m\bigr) \to 
 \alpha\Bigl(\frac{\xi +\e_1\sum_{{z}\in \rangegen}z}{\rho+\e_1\abs{\rangegen}}\Bigr) \quad  \text{as} \quad m\to \infty.\]
 Take $m\to \infty$ then   $\e_1\to 0,$ use again the continuity of $\alpha$ on $\ri \cA,$  and finally take $\e_2 \to 0$  to get that 
 \[\lim_{n\to \infty}\dfrac{ a(0,x_n)}{|x_n|_1}=\alpha(\xi),
 \]
 which contradicts \eqref{contr}. This finishes the proof of \eqref{lim:shape.aux}.

If $\P$ is ergodic under $\{T_z:z\in\rangegen\}$, then for any $\xi\in\ri\cA$, $\kS_\xi$ is trivial. Hence 
$\tilde\alpha$ is deterministic on $\cplus'\cap\ri\cA$. Its continuous extension $\alpha$ is then deterministic on $\ri\cA$. All the claims of the theorem have been proved.
 \end{proof}
 
 Next, we show that the conditions \eqref{ass:V>c} and \eqref{ass:iid} in the statement of Theorem \ref{shapetheorem} each imply the condition $\E[\overline\alpha(x)]>-\infty$ appearing in Theorem \ref{shapetheorem.aux}.

\begin{lemma}\label{alphatil-lb1}
	Assume \eqref{V>0} and that $\P\{V(\w,z)\ge c\}=1$ for some $c\in \R$ and all $z\in \range\setminus\rangez.$ Then there exists a finite positive constant $C$ such that   
	\begin{equation*}
	a(x,y)\ge -C\abs{y-x}_1 \quad \P\text{-almost surely and for all $x,y\in\Z^d$ with $y-x\in\gplus$}.
	\end{equation*}
\end{lemma}

\begin{proof}
Fix $x,y\in\Z^d$ with $y-x\in\gplus$. Let $N(n)$ be the number of steps $z\in \range\setminus \rangez$ the random walk starting at $x$ took in its first $n$ steps, i.e.~$N(n)=\sum_{k=0}^{n-1}\one\{Z_{k+1}\in\range\setminus\rangez\}$. 
Then 
	\begin{equation}\label{aux00001}
	a(x,y)\ge -\log E_x\left[e^{-cN(\tau_y)}\one\lbrace \tau_y <\infty \rbrace \right].
	\end{equation}
	By Lemma \ref{lm:aux1111} we can find $\delta>0$ and $\widehat{u}\in \R^d$ such that $z\cdot\widehat{u}\ge \delta$ for all $z\in \range\setminus\rangez$ and $z\cdot\widehat{u}=0$ for $z\in \rangez.$
	If $y-x=\sum_{{z}\in \range}\gamma_{z}z$ for some $\gamma_{z}\in \Z_+,$ then $(y-x)\cdot \widehat{u}\ge \delta \sum_{{z}\in \range\setminus \rangez}\gamma_{z}$ . This implies $N(\tau_y)\le (y-x)\cdot\widehat{u}/\delta$. Then \eqref{aux00001} implies $a(x,y)\ge -\delta^{-1}\abs{c}(y-x)\cdot\widehat{u}$, which implies the claim.
\end{proof}


\begin{lemma}\label{alphatil-lb2}
	Assume the setting of Example \ref{ex:product}. Assume also that $0\not \in \Uset$, $V$ is local, $\P$ has a finite range of dependence, and $V^-\in L^p$ for some $p>d.$ 
	Then there exists a deterministic finite positive constant $C$ such that $\P$-almost surely and for all $x\in\gplus$
		\[\varliminf_{n\to\infty}n^{-1}a(0,nx)\ge-C\abs{x}_1.\]
\end{lemma}

\begin{proof}
%
%
	Let $\delta>0$ and $\widehat{u}$ be such that $z\cdot\widehat{u}\ge \delta$ for all $z\in \range.$ Such a $\widehat{u}$ exists by the Separating Hyperplane Theorem. As was argued in the proof of  Lemma \ref{alphatil-lb1}, if $x\in\gplus$ and $x_{0:m}$ is an admissible path from $0$ to $nx$ then $m\le nx\cdot\widehat{u}/\delta\le C_2\abs{x}_1n$.
	Let $h(\w)=\max_{ z\in \range }V^-(\w,z)$ and write 
	\begin{align*}
	a(0,nx)
	&\ge -\max \Bigl\{ \sum_{i=0}^{m-1}h(T_{x_i}\w):x_0=0,x_m=nx, x_{i+1}-x_i\in \range, m\le C_2\abs{x}_1n  \Bigr\} \\
	&\ge -\max \Bigl\{ \sum_{i=0}^{C_2\abs{x}_1n-1}h(T_{x_i}\w):x_0=0,x_{i+1}-x_i\in \range\Bigr\}.
	\end{align*}
	By \cite[Lemma 3.1]{Ras-Sep-14} we see that $\P$-almost surely,
	\begin{align*}
	\varliminf_{n\to\infty}n^{-1}a(0,nx)
	&\ge-\varlimsup_{n\to \infty}n^{-1}\max \Bigl\{ \sum_{i=0}^{C_2\abs{x}_1n-1}h(T_{x_i}\w):x_0=0,x_{i+1}-x_i\in \range\Bigr\}\\
	&\ge -C_3\abs{x}_1\int_0^\infty P(h\ge s)^{1/d} ds \ge -C_3\abs{x}_1\Bigl(\E\left[\abs{h}^p\right]^{1/d}\int_1^\infty \dfrac{ds}{s^{p/d}}+\int_0^1 ds \Bigr),
	\end{align*}
as desired.
\end{proof}

Now we can prove our first main theorem.

\begin{proof}[Proof of Theorem \ref{shapetheorem}]
Lemmas \ref{alphatil-lb1} and \ref{alphatil-lb2} show that conditions  \eqref{ass:V>c} and \eqref{ass:iid} in the statement of the theorem each imply that $\P$-almost surely and for all $x\in\gplus$, $\overline\alpha(x)>-C\abs{x}_1$.
The theorem then follows directly from Theorem \ref{shapetheorem.aux}. The only claim that needs a comment is the bound in part \eqref{alpha<Cxi}. Note that there are only finitely many possible sigma algebras that appear in the conditional expectation in \eqref{atil-bded} as we vary $\xi$ over $\cplus$. First, we appeal to \eqref{atil-bded} and sum over the finitely many possible sigma algebras $\kS_\xi$ and sets $\range_\xi$ in that expression to obtain an upper bound that is uniform over $\xi \in \cplus'\cap \cA$. Note that the random constant in this upper bound is integrable. The reverse inequality for $\xi \in \gplus \cap \cA$ comes from the lower bounds in Lemmas \ref{alphatil-lb1} and \ref{alphatil-lb2}. We then extend to $\xi \in \cplus \cap \cA$ using the limit in \eqref{lima=atil}, homogeneity, and the fact that $\alpha$ is a continuous extension of $\tilde\alpha$.  

The proofs of the claims for $a_\infty$ are essentially identical once one substitutes in the appropriate definitions and thus are omitted.
\end{proof}

We close the section with the proof of the point-to-point limit.

\begin{proof}[Proof of Corollary \ref{cor:shape}]
We work with the case of $a$, the case of $a_\infty$ being identical.

%
%
If $\xi\ne0$ then $x_n/|x_n|_1\to\xi/|\xi|_1$ as $n\to\infty$ and since $\xi\in\ri\cA$ so is $\xi/|\xi_1|$ and
therefore there exists a $\delta>0$ such that $x_n/|x_n|_1\in\cA_\delta$ for $n$ large enough.

If, on the other hand, $\xi=0$ (and is in $\ri\cA$) then there exists an $\e>0$ such that $\{\zeta\in\cA:|\zeta|_1\le\e\}\subset\ri\cA$.
But then for any $\eta\in\cA$ with $|\eta|_1=1$, $\e\eta\in\ri\cA$ and hence $\eta\in\ri\cA$. 
By compactness of the unit $\ell_1$-ball in $\R^d$ we have that $\{\eta\in\cA:|\eta|_1\le1\}\subset\ri\cA_\delta$ for some $\delta>0$. Thus, $x_n/|x_n|_1\in\cA_\delta$ for all $n$.

Now, whether $\xi=0$ or not, if $|x_n|_1\to\infty$ the shape theorem \ref{shapetheorem} implies that 
$|x_n|^{-1}_1\bigl(a(0,x_n)-\alpha(x_n)\bigr)\to0$ as $n\to\infty$. Since $|x_n|_1/n\to|\xi|_1$, we see that
	 \begin{align}\label{aux000}
	 \lim_{n\to\infty}\dfrac{a(0,x_n)-\alpha(x_n)}{n} =0.
	 \end{align}
In the case when $|x_n|_1$ is bounded, we have that for any subsequence along which $x_n\ne0$, 
$|x_n|^{-1}\bigl(a(0,x_n)-\alpha(x_n)\bigr)$ is bounded and $|x_n|_1/n\to0$, as $n\to\infty$. This and the fact that 
$a(0,0)=\alpha(0)=0$ imply that \eqref{aux000} still holds.\eqref{p2plim} follows from \eqref{aux000}
by writing $\alpha(x_n)/n=\alpha(x_n/n)$ and using the continuity of $\alpha$.
\end{proof}

\section{The restricted-length polymer}\label{sec:restricted}
Before we can prove Theorem \ref{VarForm} we need a detour into restricted-length random polymers. For $n\in \N$ let 
\begin{equation}\label{def:Dn}
D_n = \Bigl\{ x\in\gplus:\exists b_z\in\Z_+,\,z\in\range, \text{ with } \sum_{{z}\in \range}b_z=n \text{ and } x=\sum_{{z}\in \range}b_z z \Bigr\}.
\end{equation}
For $y-x\in D_n$ let
\[G_{x,(n),y} = \log E_x\Bigl[  \exp\Bigl\{ - \sum_{k=0}^{n-1}  V(T_{X_k}\w,Z_{k+1})\Bigr\} \one_{\lbrace X_n=y \rbrace} \Bigr]\quad\text{and}\quad
G^\infty_{x,(n),y}=\max_{\substack{x_{0:n}:x_0=x,\\\ \ \ \ \ x_n=y}}\Bigl\{-\sum_{k=0}^{n-1}V(T_{x_k}\w,z_{k+1})\Bigr\}.\]

The following theorem follows from \cite[Theorem 2.2]{Ras-Sep-14} and  \cite[Theorem 2.4]{Geo-Ras-Sep-16}.

\begin{theorem}\label{thm:p2p}
	Assume $V^+(\w,z)\in \cL_{z,\range}$ for all $z\in\range\setminus\{0\}$ and $V^+(\w,0)\in\cL_{\hat z,\range}$ for some $\hat z\in\range\setminus\{0\}$. Assume $\P$ is ergodic under the group of shifts $\{T_x:x\in\cG\}$.
Then the following hold.
	 \begin{enumerate}[label={\rm(\roman*)}, ref={\rm\roman*}] \itemsep=3pt 
	\item\label{p2p}For $\P$-almost every $\w$ and simultaneously for all $\xi\in \Uset$ the limits
	\begin{equation}\label{eq:p2p}
	\Lapp(\w,\xi)=\lim_{n\to \infty} \frac{G_{0,(n),\tilde x_{n}(\xi)}}{n}\quad\text{and}\quad
	\gpp(\w,\xi)=\lim_{n\to \infty} \frac{G^\infty_{0,(n),\tilde x_{n}(\xi)}}{n}
	\end{equation}
	exist in $(-\infty,\infty]$. Here, $\tilde x_n(\xi)\in D_n$ is defined in  \cite[Equation (2.1)]{Ras-Sep-14}. It satisfies $\tilde x_n(\xi)/n\to\xi$ as $n\to\infty$. 
	\item\label{p2l} Fix $h\in\R^d$. For $\P$-almost every $\w$ the limits
    \begin{equation}\label{eq:p2l}
	\Lapl(\w,h)=\lim_{n\to \infty} \frac{1}{n}\log\sum_{x\in D_n}e^{G_{0,(n),x}+h\cdot x}\quad\text{and}\quad
    \gpl(\w,h)=\lim_{n\to \infty} \frac{1}{n}\max_{x\in D_n}\{G^\infty_{0,(n),x}+h\cdot x\} 
    \end{equation}
	exist in $(-\infty,\infty]$ and satisfy
	\begin{align*}
    \Lapl(h)=\sup_{\xi\in\Uset}\{\Lapp(\xi)+h\cdot\xi\}
    \quad\text{and}\quad
    \gpl(h)=\sup_{\xi\in\Uset}\{\gpp(\xi)+h\cdot\xi\}.
	\end{align*}
	\end{enumerate}
\end{theorem}

\begin{remark}
\cite[Theorem 2.2]{Ras-Sep-14} requires that $\max_{z\in\range}\abs{V(\w,z)}\in\cL_{\bar z,\range}$ for each $\bar z\in\range\setminus\{0\}$, but examining the proof shows that it is in fact enough to assume that $V^+(\w,z)\in\cL_{z,\range}$ for all $z\in\range\setminus\{0\}$ and $V^+(\w,0)\in\cL_{\hat z,\range}$ for some $\hat z\in\range\setminus\{0\}$.
\end{remark}

\begin{remark}
	 Under the conditions of Theorem \ref{shapetheorem}, $\Lapp$ and $\gpp$ are finite. 
	 See \cite[Remark 2.3]{Ras-Sep-14} and \cite[Remarks 2.5 and 2.6]{Geo-Ras-Sep-16}.
\end{remark}

In fact, a shape theorem similar to the one in Theorem \ref{shapetheorem} holds for $G_{0,(n),x}$ and $G^\infty_{0,(n),x}$.
Given $\delta>0$ and a face $\Usetgen$ of $\Uset$ let \[\Usetgen_\delta=\bigl\lbrace \xi \in \Usetgen:\text{dist}\bigl(\xi,\Usetgen\setminus\ri\Uset\bigr)\ge\delta \bigr\rbrace \qquad \text{ and } \qquad \rangegen=\Usetgen\cap \range.\]
 

\begin{theorem}\label{th:Gshape}
	Fix a face $\Usetgen$ of $\Uset$ that is not a singleton. Assume $V^+(\w,z)\in \cL_{z,\rangegen}$ for each $z\in\rangegen\setminus\{0\}$ and $V^+(\w,0)\in\cL_{\hat z,\rangegen}$ for some $\hat z\in\rangegen\setminus\{0\}$.
	Assume $\P$ is ergodic under $\{T_x:x\in\cG(\rangegen)\}$.  Then for any $\delta>0$ we have $\P$-almost surely
	\begin{equation}\label{eq:Gshape}
	\varlimsup_{n\to \infty}\max_{x\in n\Usetgen_\delta\cap  D_n} \frac{\bigl|G_{0,(n),x}-n\Lapp\bigl(\frac{x}{n}\bigr)\bigr|}{n}= 0\quad\text{and}\quad
	\varlimsup_{n\to \infty}\max_{x\in n\Usetgen_\delta\cap  D_n} \frac{\bigl|G^\infty_{0,(n),x}-n\gpp\bigl(\frac{x}{n}\bigr)\bigr|}{n}= 0.
	\end{equation}
	Consequently, Theorem \ref{thm:p2p}\eqref{p2p} holds with $\tilde x_n(\xi)$ replaced by any sequence $x_n\in D_n$ satisfying $x_n/n\to\xi$, on a single event of full probability. 
\end{theorem}

\begin{remark}
The point of using the sets $\Usetgen_\delta$ is to stay uniformly away from the places where $\Lapp$ may not be continuous.  Theorem 3.2 in \cite{Ras-Sep-14} gives conditions under which $\Lapp$ is continuous up to the boundary. The same result should hold for $\gpp$. In these cases, one can strengthen the above shape theorem and include some or all of the boundary. For example, if $\P$ is i.i.d., $V$ is local, $V\in L^p$ with $p>d$, and $0\not\in \Uset$ then the shape theorem holds on all of $\Uset$. The same holds if $V$ is bounded above and $0\in\ri\Uset$.  When $0$ is on the relative boundary of $\Uset$ the shape theorem holds if one stays uniformly away from $\Uset_0$, the unique face of $\Uset$ that contains $0$ in its relative interior.
\end{remark}

We do not need Theorem \ref{th:Gshape} for our proofs of Theorems \ref{shapetheorem} and \ref{VarForm}. However, this shape theorem is of independent interest to the field. Hence, we give a proof of it in Appendix \ref{app:Gshape}.\medskip

The next lemma connects restricted-length and unrestricted-length quantities.

\begin{lemma}\label{-a>La}
Assume the setting of Theorem \ref{shapetheorem}. Assume also that $\P$ is ergodic under $\{T_x:x\in\cG\}$. Then with $\P$-probability one we have
for each $\xi\in\cC_+$ and each $s>0$ such that $\xi/s\in\Uset$, $s\Lapp(\xi/s)\le-\alpha(\xi)$ and $s\gpp(\xi/s)\le-\alpha_\infty(\xi)$.
\end{lemma}

\begin{proof}
We work with the case of $\Lapp$ and $\alpha$, the case of $\gpp$ and $\alpha_\infty$ being identical. 

Consider $\xi$ and $s$ as in the claim. 
Then on the event $\{X_{\fl{ns}}=\tilde x_{\fl{ns}}(\xi/s)\}$ we have
$\tau_{\tilde x_{\fl{ns}}(\xi/s)}\le\fl{ns}<\infty$ and Lemma \ref{lm:loops} tells us that $Z_k\in \rangez$ for $\tau_{\tilde x_{\fl{ns}}(\xi/s)}\le k<\fl{ns}$ if this set is non-empty. If $0\not\in\Uset$ then there is no such $k$ and if $0\in\Uset$ then 
condition \eqref{V>0} implies that  $V(T_{X_k}\w,Z_{k+1})\ge 0$ for all such $k$. Consequently,
		\begin{align*}
		&E_0\Bigl[  \exp\Bigl\{ - \sum_{k=0}^{\fl{ns}-1}  V(T_{X_k}\w,Z_{k+1})\Bigr\} \one{\{ X_{\fl{ns}}=\tilde x_{\fl{ns}}(\xi/s) \}} \Bigr]\\ 
		&\qquad\qquad\le E_0\biggl[  \exp\biggl\{ - \sum_{k=0}^{\tau_{\tilde x_{\fl{ns}}(\xi/s)}-1}  V(T_{X_k}\w,Z_{k+1})\biggr\} \one{\{\tau_{\tilde x_{\fl{ns}}(\xi/s)}<\infty\}} \biggr].
		\end{align*}
		Take a log, divide by $n$,  send $n\to\infty$, and apply Theorems \ref{thm:p2p} and \ref{shapetheorem} to conclude $s\Lapp(\xi/s)\le -\alpha(\xi)$.
\end{proof}

\begin{remark}
The above shows that $\alpha(\xi)\le-\sup_{s>0} \{s\Lapp(\xi/s)\}$. We believe that in fact
    \[\alpha(\xi)=-\sup_{s>0} \{s\Lapp(\xi/s)\}\]
and that the supremum is attained. A similar statement should hold for $\alpha_\infty$ and $\gpp$. We leave this to future work, 
as we do not need this for our results in this paper.
\end{remark}

We will need the following consequence of  \eqref{eq:p2l}.

\begin{theorem}\label{newshapetheorem3}
	Let $V^+(\w,z)\in\cL_{z,\range}$ for each $z\in\range\setminus\{0\}$ and $V^+(\w,0)\in\cL_{\hat z,\range}$ for some $\hat z\in\range\setminus\{0\}$. Assume also that $\P$ is ergodic under the group of shifts $\{T_x:x\in\cG\}$. 
	Fix $h\in\R^d$. Then for $\P$-almost every $\w$
	\begin{align*}
    &\lim_{n \to \infty} \frac{1}{n} \log \sum_{k=0}^{n-1}E_0\Bigl[\exp\Bigl\{-\sum_{i=0}^{k-1}V(T_{X_i}\w,Z_{i+1})+h\cdot X_k \Bigr\} \Bigr]  
	=\sup_{\substack{0<s\le 1 \\  \xi:\xi/s\in\Uset}  }\Bigl\lbrace s\Lapp\Bigl( \frac{\xi}{s}  \Bigr)  +h\cdot\xi \Bigr\rbrace\quad\text{and}\\
    &\lim_{n \to \infty} \frac{1}{n} \max_{0\le k\le n-1}\max_{x_{0:k}:x_0=0}\Bigl\{-\sum_{i=0}^{k-1}V(T_{x_i}\w,z_{i+1})+h\cdot x_k \Bigr\}  
	=\sup_{\substack{0<s\le 1 \\  \xi:\xi/s\in\Uset}  }\Bigl\lbrace s\gpp\Bigl( \frac{\xi}{s}  \Bigr)  +h\cdot\xi \Bigr\rbrace.
	\end{align*}
\end{theorem}

\begin{proof}
We prove the first limit, the second being similar.
Observe that 
    \begin{align}\label{absLa}
    \sup_{\substack{0<s\le 1 \\  \xi:\xi/s\in\Uset}  }\Bigl\lbrace s\Lapp\Bigl( \frac{\xi}{s}  \Bigr)  +h\cdot\xi \Bigr\rbrace
    =\sup_{0<s\le 1}s\sup_{\xi\in\Uset}\{\Lapp(\xi)+h\cdot\xi\}
    =\max(0,\Lapl(h)).
    \end{align}

Next, write 
\begin{align*}
 &\frac{1}{n} \log\sum_{k=0}^{n-1}E_0\Bigl[\exp\Bigl\{-\sum_{i=0}^{k-1}V(T_{X_i}\w,Z_{i+1})+h\cdot X_k \Bigr\} \Bigr] 
   =  \frac{1}{n} \log \sum_{k=0}^{n-1}  \sum_{x\in D_k} e^{G_{0,(k),x}+h\cdot x}\notag\\
&\qquad\le \frac{1}{n}\log n +  \frac{1}{n} \max_{0\le k\le n-1}\log\sum_{x\in D_k} e^{G_{0,(k),x}+h\cdot x}\,.
    \end{align*}
Together with \eqref{eq:p2l} this gives
\[ \varlimsup_{n \to \infty} \frac{1}{n} \log \sum_{k=0}^{n-1}E_0\Bigl[\exp\Bigl\{-\sum_{i=0}^{k-1}V(T_{X_i}\w,Z_{i+1})+h\cdot X_k \Bigr\} \Bigr]  \le\Lapl(h).\] 
This and \eqref{absLa} imply the upper bound
	\begin{align*}
    \varlimsup_{n \to \infty} \frac{1}{n} \log \sum_{k=0}^{n-1}E_0\Bigl[\exp\Bigl\{-\sum_{i=0}^{k-1}V(T_{X_i}\w,Z_{i+1})+h\cdot X_k \Bigr\} \Bigr]  
	\le\sup_{\substack{0<s\le 1 \\  \xi:\xi/s\in\Uset}  }\Bigl\lbrace s\Lapp\Bigl( \frac{\xi}{s}  \Bigr)  +h\cdot\xi \Bigr\rbrace.
	\end{align*}

For the other bound observe that replacing the sum by the $k=n-1$ term gives
	\begin{align*}
    &\varliminf_{n \to \infty} \frac{1}{n} \log \sum_{k=0}^{n-1}E_0\Bigl[\exp\Bigl\{-\sum_{i=0}^{k-1}V(T_{X_i}\w,Z_{i+1})+h\cdot X_k \Bigr\} \Bigr]  \\
	&\qquad\ge
    \varliminf_{n \to \infty} \frac{1}{n-1} \log E_0\Bigl[\exp\Bigl\{-\sum_{i=0}^{n-2}V(T_{X_i}\w,Z_{i+1})+h\cdot X_k \Bigr\} \Bigr]=\Lapl(h)  
	\end{align*}
    and similarly, replacing the sum by the $k=0$ term gives
	\begin{align*}
    \varliminf_{n \to \infty} \frac{1}{n} \log \sum_{k=0}^{n-1}E_0\Bigl[\exp\Bigl\{-\sum_{i=0}^{k-1}V(T_{X_i}\w,Z_{i+1})+h\cdot X_k \Bigr\} \Bigr] \ge0.
    \end{align*}
Together with \eqref{absLa} these two lower bounds give the desired lower bound that completes the proof of the theorem.
\end{proof}

\section{Proof of Theorem \ref{VarForm}}
The lower bound comes by an application of the shape theorem \eqref{p2plim} and a perturbation of the potential by the cocycle. 
The upper bound comes by constructing an approximately optimal cocycle and then extracting an optimal cocycle from a converging subsequence.\smallskip

We again work with the case of $\alpha$ and leave it to the reader to check that the case of $\alpha_\infty$ works similarly.\smallskip

Fix a face $\cA\ne\{0\}$ of $\cC_+$, $B\in\cK_\cA$, and rationals $\gamma_z\in\Q_+$, where $z\in\range'=\range\cap\cA$.
Take $\ell\in\N$ such that $\ell\gamma_z\in\Z_+$ for all $z\in\range'$.
Write $\xi=\sum_{z\in\range'}\gamma_z z$ and $x_{n\ell}=\sum_{z\in\range'}n\ell\gamma_z z=n\ell\xi$, for $n\in\N$.
Then 
	\[B(\w,0,x_{n\ell})=\sum_{i=0}^{n-1} B(T_{ix_\ell}\w,0,x_\ell),\]
and the ergodic theorem implies that with $\P$-probability one,
	\[\lim_{n\to\infty}(n\ell)^{-1}B(0,x_{\ell n})=\ell^{-1}\E[B(0,x_\ell)\,|\,\cI_{x_\ell}],\]
where we recall that for $x\in\cG$, $\cI_x$ is the $\sigma$-algebra of $T_x$-invariant events.
By \cite[Lemma B.4]{Jan-Ras-18-arxiv}, we have $\P$-almost surely $\E[B(0,x_\ell)\,|\,\cI_{x_\ell}]=-h(B)\cdot x_\ell$ and hence
the above limit equals $-h(B)\cdot\xi$.
Now by \eqref{p2plim} we have with $\P$-probability one
\begin{align*}
\alpha(\xi)
&=\lim_{n\to\infty}\frac{-1}{n\ell}\log E_0\Bigl[ \exp\Bigl\{-\sum_{k=0}^{\tau_{x_{n\ell}}-1} V(T_{X_k}\w, Z_{k+1})\Bigr\} \one_{\lbrace \tau_{x_{n\ell}}<\infty \rbrace}  \Bigr]\\
&=\lim_{n\to\infty}\frac{-1}{n\ell}\log E_0\Bigl[ \exp\Bigl\{-\sum_{k=0}^{\tau_{x_{n\ell}}-1} V(T_{X_k}\w, Z_{k+1})-B(\w,0,x_{n\ell})\Bigr\} \one_{\lbrace \tau_{x_{n\ell}}<\infty \rbrace}  \Bigr]+h(B)\cdot\xi\\
&=\lim_{n\to\infty}\frac{-1}{n\ell}\log E_0\Bigl[ \exp\Bigl\{-\sum_{k=0}^{\tau_{x_{n\ell}}-1} \bigl(V(T_{X_k}\w, Z_{k+1})+B(T_{X_k}\w,0,Z_{k+1})\bigr)\Bigr\} \one_{\lbrace \tau_{x_{n\ell}}<\infty \rbrace}  \Bigr]
+h(B)\cdot\xi.
\end{align*}
If $B\in\cK^+_\cA(V)$, then the expected value on the last line of the above display is bounded above by 
$1$. To see this consider the Markov chain  that moves from $x\in\cG(\rangegen)$ to $x+z$, $z\in\rangegen$, 
with probability 
$e^{-V(T_x\w,z)-B(T_x\w,0,z)}p(z)$ and moves from $x$ to a cemetery state $\Delta$ with the remaining probability 
$1-\sum_{z\in\rangegen}e^{-V(T_x\w,z)-B(T_x\w,0,z)}p(z)$. 
Once at $\Delta$ the chain remains there forever.
Then the expectation in question is the same as the probability this Markov chain ever reaches $x_{n\ell}$.

We have thus shown that for each fixed $B\in\cK^+_\cA(V)$ and $\xi\in\cA$, \eqref{alpha>h.xi} holds with probability one. This inequality also holds with probability one, simultaneously for a countable dense set of $\xi\in\cA$.
Since $h(B)\cdot\xi$ is continuous in $\xi$ and $\alpha$ is continuous in $\xi$ on the interior of each face of $\cC_+$, we conclude that for each fixed $B\in\cK^+_\cA(V)$ we have that with $\P$-probability one \eqref{alpha>h.xi} holds for all $\xi\in\cA$.\smallskip


Next, we prove \eqref{alpha-var}. It is enough to work with $\cA=\cC_+$. The proof for the case of a different face is the same, after replacing $\cC_+$ by $\cA$ everywhere. We therefore drop the $\cA$ from the indices of $\cK_\cA$ and $\cK^+_\cA(V)$ and now assume $\P$ is ergodic under $\{T_z:z\in\range\}$. 

First, observe that since $\P$ is ergodic under the shifts $\{T_x:x\in\cG\}$, $\alpha$ is deterministic on $\ri\cC_+$.
Extend $\alpha$ to all of $\R^d$ by setting $\alpha(\xi)=\infty$ for $\xi\not\in\cC_+$. For $h\in\R^d$ let
    \[\alpha^*(h)=\sup\{h\cdot\xi-\alpha(\xi):\xi\in\R^d\}\]
be the convex conjugate of $\alpha$. Note that $\alpha(0)=0$ implies $\alpha^*(h)\ge0$. Furthermore, 
for any $s>0$
    \[\alpha^*(h)=s\sup\{h\cdot\xi/s-\alpha(\xi/s):\xi\in\R^d\}=s\alpha^*(h).\]
Consequently, 
    \[\alpha^*(h)=\begin{cases}0&\text{if }h\cdot\xi\le\alpha(\xi)\quad\forall\xi\in\cC_+,\\ \infty&\text{otherwise.}\end{cases}\]

Since $\alpha$ is convex \cite[Theorem 4.17]{Ras-Sep-15-ldp} says that the bi-conjugate
    \[\alpha_\cA^{**}(\xi)=\sup\{h\cdot\xi-\alpha^*(h):h\in\R^d\}\]
is the same as the lower semicontinuous regularization of $\alpha$, 
which matches $\alpha$ on $\ri\cC_+$. Therefore, for each 
$\xi\in\ri\cC_+$ and each $j\in\N$ there exists an $h_j\in\R^d$ such that 
$\alpha^*(h_j)=0$
and
    \begin{align}\label{h>a}
    h_j\cdot\xi\ge\alpha(\xi)-1/j.
    \end{align}

By  Theorem \ref{newshapetheorem3} and Lemma \ref{-a>La} we have $\P$-almost surely
    \begin{align*}
    \varlimsup_{n\to\infty} n^{-1}\log\sum_{k=0}^{n-1} E_0\Bigl[\exp\Bigl\{-\sum_{i=0}^{k-1}V(T_{X_i}\w,Z_{i+1})+h_j\cdot X_k\Bigr\}\Bigr]
    &\le\sup_{\substack{0<s\le1\\ \zeta/s\in\Uset}}\bigl\{h_j\cdot\zeta+s\Lapp(\zeta/s)\bigr\}\\
    &\le\sup_{\zeta\in\cC_+}\{h_j\cdot\zeta-\alpha(\zeta)\}\\
    &=\alpha^*(h_j)=0.
    \end{align*}
    

	Define 
	\begin{equation}
	g_j(\w)= \log \Bigl(1+\sum_{n\ge 1}e^{-n/j}\sum_{k=0}^{n-1}E_{0}^\w\Bigl[\exp\Bigl\{-\sum_{i=0}^{k-1}V(T_{X_i}\w,Z_{i+1})+h_j\cdot X_k\Bigr\}\Bigr]\Bigr).\label{eq:gdef}
	\end{equation}

	We just showed that the inner sum in \eqref{eq:gdef} grows subexponentially in $n$ and hence $g_j\ge0$ is finite $\P$-almost surely.
	Furthermore,
	\begin{align*}
	e^{g_j(\w)} 
	&= 1+\sum_{n\ge1}e^{-n/j}\Bigl(1+\sum_{k=1}^{n-1}\sum_{{z}\in \range}p(z)e^{-V(\w,z)+h_j\cdot z}\,
    E_0^{T_z\w}\Bigl[e^{-\sum_{i=0}^{k-2}V(T_{X_i}T_z\w,Z_{i+1})+h_j\cdot X_{k-1}}\Bigr] \Bigr)\\
	& = \frac{1}{1-e^{-1/j}}+e^{-1/j}\sum_{{z}\in \range}p(z)e^{-V(\w,z)+h_j\cdot z +g_j(T_z\w)}.
	\end{align*}

    \begin{remark}
    In zero temperature, i.e.\ for the case of $\alpha_\infty$, the analogous definition would be
	\begin{align*}
	    g_{j,\infty}(\w)
	    &= \max_{n \geq 1}\max_{0 \leq k \leq n-1}\max_{x_{0:k}:x_0=0}\Bigl\{ -\sum_{i=0}^{k-1}V(T_{x_i}\w,z_{i+1}) +h_j\cdot x_k - n/j\Bigr\}\\
	    &= \max_{z \in \range}\bigl\{-V(\w,z) + h_j \cdot z + g_{j,\infty}(T_z \w)\bigr\}  - 1/j.
	\end{align*}
    \end{remark}
	
	Setting $B_j(\w,x,y)= g_j(T_x\w)- g_j(T_y\w )-h_j\cdot(y-x)$, we have $\P$-almost surely
	\begin{equation}\label{B-up}
	\sum_{{z}\in \range} p(z)e^{-V(\w,z) -B_j(\w,0,z)}\le e^{1/j}.
	\end{equation}
    This implies that  for each $z\in \range$,
        \[B_j(\w,0,z)\ge\log p(z)-V(\w,z)-1/j.\]
    Since $V^+\in L^{1}$  we see that $B_j^-(0,z)$ is uniformly integrable.
    Since $\E[B_j(0,z)]=-h_j\cdot z$, we have $\E[B_j^+(0,z)]=\E[B_j^-(0,z)]-h_j\cdot z$.
    Since $\xi\in\ri\cC_+$ we can find $\gamma_z>0$, $z\in\range$, such that $\xi=\sum_{z\in\range}\gamma_z z$. Then for any $z\in\range$,
        \begin{align*}
        \gamma_z\,\E[B_j^+(0,z)]
        \le\sum_{\bar z\in\range}\gamma_{\bar z}\,\E[B_j^+(0,\bar z)]
        =\sum_{\bar z\in\range}\gamma_{\bar z}\,\E[B_j^-(0,\bar z)]-h_j\cdot\xi
        \le\sum_{\bar z\in\range}\gamma_{\bar z}\E[B_j^-(0,\bar z)]
        -\alpha(\xi)+1/j.
        \end{align*}
    Therefore, $\E[B_j^+(0,z)]$ is uniformly bounded for each $z\in\range$. By \cite[Lemma 4.3]{Kos-Var-08} we can write
    \[B_j^+(0,z) =\hat{B}^{+}_j(0,z) +R_j(z) \]
    where, along a subsequence,  $\hat{B}^{+}_j(0,z)$ is uniformly integrable and $R_j(z) \ge 0$ converges to 0 in $\P$-probability. Extract a further subsequence $\hat{B}^{+}_{j_\ell}(0,z)$ of $\hat{B}^{+}_j(0,z)$ such that
    $\tilde{ B}_{j_\ell}(0,z)= \hat{B}^{+}_{j_\ell}(0,z)-B^{-}_{j_\ell}(0,z)$ is weakly convergent in $L^1(\P)$ to some ${B}(0,z)$, and  $R_{j_\ell}(z)$ converges $\P$-almost surely to 0, for all $z\in\range$. Abbreviate $j_\ell$ as $j$. 
    By \cite[Theorem 3.12]{Rud-91},
    ${B}(0,z)$ is in the strong $L^1(\P)$-closure of the convex hull of $\lbrace \tilde{B}_j(0,z):j\ge k \rbrace$ for any $k \geq 1$. This means that there exists a 
    sequence of finite convex combinations 
    $\tilde{G}_j(0,z) = \sum_{k=j}^{\infty}\delta_{j,k} \tilde{B}_k(0,z)$ that converges to $ { B}(0,z)$ strongly  in $L^1(\P)$.
    Up to a further subsequence, $\tilde{G}_j(0,z)$ converges $\P$-almost surely to ${B}(0,z)$, for all $z\in\range$. Consequently, 
        \[G_j(0,z)=\sum_{k=j}^\infty\delta_{j,k} B_k(0,z)=\tilde G_j(0,z)+\sum_{k=j}^\infty\delta_{j,k} R_k(z)\] 
    also converges $\P$-almost surely to $B(0,z)$.
    
    Since $B_j$ is a covariant cocycle we have $\P$-almost surely and for any $z_1,z_2\in\range$
        \[B_j(\w,0,z_1)+B_j(T_{z_1}\w,0,z_2)=B_j(\w,0,z_2)+B_j(T_{z_2}\w,0,z_1).\]
    This {\it one cell} cocycle property transfers to $G_j$ and thus to the limit $B$. Define \[B(\w,x,x+z)=-B(\w,x+z,x)=B(T_x\w,0,z)\] 
    for $x\in\cG$ and $z\in\range$. Enumerate $\range=\{z_1,\dotsc,z_M\}$. 
    For $x,y\in\cG$ write $y-x=\sum_{i=1}^M b_i z_i$ with $b_i\in\Z$, 
    and define
    \[B(\w,x,y)=\sum_{i=1}^M\sum_{j=1}^{b_i} B\Bigl(\w,x+\sum_{r=1}^{i-1}b_r z_r+(j-1)z_i,x+\sum_{r=1}^{i-1}b_r z_r+jz_i\Bigr).\]
    Due to the one cell cocycle property, this definition does not depend on the choice of the coefficients $b_i$. It is also immediate that now $B$ is an $L^1$ covariant cocycle. Furthermore, by \eqref{B-up} and Jensen's inequality
    \[\sum_{z\in\range}p(z)e^{-V(\w,z)-G_j(\w,0,z)}
    \le\sum_{k=j}^\infty\delta_{j,k}\sum_{z\in\range}p(z)e^{-V(\w,z)-B_k(\w,0,z)}\le \sum_{k=j}^\infty \delta_{j,k}e^{1/k}\le e^{1/j}.\]
Taking $j\to\infty$ shows that $B\in\cK^+(V)$. Also,
    \[h_j\cdot z=-\E[B_j(0,z)]=-\E[\tilde B_j(0,z)+R_j(z)]\le-\E[\tilde B_j(0,z)].\]
    This and \eqref{h>a} imply that if we write $\xi=\sum_{z\in\range}\gamma_z z$ with $\gamma_z\ge0$ then
    \[-\sum_{z\in\range}\gamma_z\,\E[\tilde G_j(0,z)]\ge\sum_{k=j}^\infty\delta_{j,k}\,h_k\cdot\xi
    \ge\alpha(\xi)-\sum_{k=j}^\infty\frac{\delta_{j,k}}k\ge\alpha(\xi)-1/j.\]
    Taking $j\to\infty$ we find that 
    \[h(B)\cdot\xi=\sum_{z\in\range}\gamma_z\,h(B)\cdot z=-\sum_{z\in\range}\gamma_z\,\E[B(0,z)]\ge\alpha(\xi).\]
    This implies \eqref{alpha-var}. Together with \eqref{alpha>h.xi} it also shows that in fact $\alpha(\xi)=h(B)\cdot\xi$.  Theorem \ref{VarForm} is proved for the case of $\alpha$. 
    The case of $\alpha_\infty$ is almost identical after the appropriate definitions are substituted.\hfill\qed

\section{Proof of Theorem \ref{shape-green}}
First, we relate the Green's function $g$ to $a$.

\begin{lemma}
	For $x,y$ with $y-x\in \gplus,$
	\begin{equation}\label{eq:aux1010}
	a(x,y)+\log g(x,y)=\log g(y,y).
	\end{equation}
\end{lemma}

 \begin{proof}
Since $a(x,x)=0$ the identity is clear if $y=x$. Assume $y\ne x$. Then applying the Markov property in the first equality we have
\begin{align*}
 &g(x,y)
 = \sum_{m=1}^\infty  \sum_{\ell=1}^{m} E_x\Bigl[  e^{ -\sum_{k=0}^{\ell-1} V(T_{X_k}\w,Z_{k+1})}  \one_{\lbrace\tau_y=\ell \rbrace } \Bigr]\times E_y \Bigl[ e^{ - \sum_{k=0}^{m-\ell-1} V(T_{X_k}\w,Z_{k+1})} \one_{\lbrace X_{m-\ell}=y\rbrace} \Bigr]\\
&=\sum_{\ell =1}^\infty   E_x\Bigl[  \exp\Bigl\{ -\sum_{k=0}^{\ell-1} V(T_{X_k}\w,Z_{k+1})\Bigr\} \one_{\lbrace\tau_y=\ell \rbrace} \Bigr]\times \sum_{m=\ell }^{\infty} E_y \Bigl[ \exp\Bigl\{ -\!\!\sum_{k=0}^{m-\ell-1} V(T_{X_k}\w,Z_{k+1})\Bigr\} \one_{\lbrace X_{m-\ell}=y\rbrace} \Bigr] \\
&=\sum_{\ell =1}^\infty   E_x\Bigl[  \exp\Bigl\{ -\sum_{k=0}^{\ell-1} V(T_{X_k}\w,Z_{k+1})\Bigr\} \one_{\lbrace \tau_y=\ell \rbrace } \Bigr] \times \sum_{m=0}^{\infty} E_y \Bigl[ \exp\Bigl\{ - \sum_{k=0}^{m-1} V(T_{X_k}\w,Z_{k+1})\Bigr\} \one_{\lbrace X_{m}=y\rbrace } \Bigr] \\
&= e^{-a(x,y)}\times g(y,y).
\end{align*}
The claim follows.
 \end{proof}
 
 Let $\sigma_k$ be the time of $k$-th return of the reference random walk to its starting point.

\begin{lemma}
Assume \eqref{V>0}. Then $\P$-almost surely, 
for all $y\in \Z^d$,
    \begin{align}\label{g-bound}
    1\le g(y,y)=\dfrac{1}{1-E_y\bigl[\exp \bigl\{ - \sum_{k=0}^{\sigma_1-1}V(T_{X_k}\w,Z_{k+1})\bigr\}  \one_{\lbrace  \sigma_1<\infty \rbrace}\bigr] }\le 1/P_y(\sigma_1=\infty).
    \end{align}
\end{lemma}

\begin{proof}The bound $g(y,y)\ge1$ is clear. For the other bound without loss of generality, we consider $y=0$. By decomposing into the number of returns by time $m$ we can write  \begin{align*}
g(0,0)
&= 1+\sum_{m=1}^{\infty} E_0 \Bigl[ \exp\Bigl\{ - \sum_{k=0}^{m-1} V(T_{X_k}\w,Z_{k+1})\Bigr\}\one_{\lbrace X_{m}=0\rbrace } \Bigr] \\
&=  1+E_0\Bigl[ \sum_{m=1}^\infty \sum_{i=1}^m  \exp \Bigl\{- \sum_{k=0}^{\sigma_1-1}V(T_{X_k}\w,Z_{k+1})\Bigr\} \cdots \exp \Bigl\{-  \sum_{k=\sigma_{i-1}}^{\sigma_i-1}V(T_{X_k}\w,Z_{k+1})\Bigr\} \one\{\sigma_i=m\} \Bigr] \\ 
&= 1+ E_0\Bigl[ \sum_{i=1}^\infty \exp \Bigl\{ - \sum_{k=0}^{\sigma_1-1}V(T_{X_k}\w,Z_{k+1})\Bigr\} \cdots \exp \Bigl\{ - \sum_{k=\sigma_{i-1}}^{\sigma_i-1}V(T_{X_k}\w,Z_{k+1})\Bigr\}  \sum_{m=i}^\infty  \one_{\lbrace  \sigma_i=m \rbrace} \Bigr].
\end{align*}
Since $\sum_{m=i}^\infty  \one_{\lbrace  \sigma_i=m \rbrace}= \one_{\lbrace  \sigma_i<\infty \rbrace}  =  \one_{\lbrace  \sigma_1<\infty \rbrace}  \one_{\lbrace  \sigma_2-\sigma_1<\infty \rbrace} \cdots  \one_{\lbrace  \sigma_k-\sigma_{k-1}<\infty \rbrace}$ we have by the Markov property
\begin{align*}
g(0,0)&=1+ E_0\Bigl[ \sum_{i=1}^\infty \exp \Bigl\{ - \sum_{k=0}^{\sigma_1-1}V(T_{X_k}\w,Z_{k+1})\Bigr\} \one_{\lbrace  \sigma_1<\infty \rbrace}\\
&\qquad\qquad\qquad\qquad\qquad\cdots \exp \Bigl\{ - \sum_{k=\sigma_{i-1}}^{\sigma_i-1}V(T_{X_k}\w,Z_{k+1})\Bigr\}   \one_{\lbrace  \sigma_i-\sigma_{i-1}<\infty \rbrace}  \Bigr]\notag\\
&=1+\sum_{i=1}^\infty E_0\Bigl[ \exp \Bigl\{- \sum_{k=0}^{\sigma_1-1}V(T_{X_k}\w,Z_{k+1})\Bigr\} \one_{\lbrace  \sigma_1<\infty \rbrace} \Bigr]^k.
\end{align*}

If $P_0(\sigma_1=\infty)=1$, which includes the case $0\not\in\Uset$, then $g(0,0)=1$ and the claim of the lemma holds.  If, on the other hand, $P_0(\sigma_1=\infty)<1$, then it must be that $0\in \Uset$ and on the event $\{\sigma_1<\infty\}$ we have  $X_{\sigma_1}=0$ and Lemma \ref{lm:loops} tells us that $Z_k\in \rangez$ for $k<\sigma_1$. Condition \eqref{V>0} implies then that  $V(T_{X_k}\w,Z_{k+1})\ge 0$ for all such $k$. Consequently,  $E_0\bigl[ e^{- \sum_{k=0}^{\sigma_1-1}V(T_{X_k}\w,Z_{k+1})} \,\one_{\lbrace  \sigma_1<\infty \rbrace} \bigr]\le P_0(\sigma_1<\infty)<1$ and \eqref{g-bound} follows.
\end{proof}

The next lemma follows the idea in \cite[Lemma 5]{Zer-98-aap}.

\begin{lemma}
Assume \eqref{V>0} and that $\P\{V(\w,z)>0\}>0$ for some $z\in\rangez$. 
Then $\P$-almost surely
    \begin{align}\label{g/y}
    \varlimsup_{\abs{x}_1\to\infty}\frac{\abs{\log g(x,x)}}{\abs{x}_1}=0.
    \end{align}
\end{lemma}

\begin{proof}
The previous lemma shows that if $P_0$ is transient, then $\log g(y,y)$ is bounded. \eqref{g/y} holds in this case. Assume therefore that $P_0$ is recurrent. In particular, $0\in\Uset$.  

Let $z_0\in\rangez$ and $\e>0$ be such that $\P\{V(\w,z_0)\ge\e\}>0$. 
Note that if $x\in\cG_+(\rangez)\setminus\{0\}$, then  Lemma \ref{boundednessofsteps} says that we can write $x=\sum_{z\in\rangez}\gamma_z z$ with $\gamma_z\in\Z_+$ and $\gamma_z\le C\abs{x}_1$ for all $z\in\rangez$. This produces an admissible path $x_{0:n}$ from $0$ to $x$ of length $n=\sum_{z\in\rangez}\gamma_z\le C\abs{\rangez}\cdot\abs{x}_1$. 
Let $x_{n+1}=x+z_0$. Similarly, we can get an admissible path $x_{n+1:m}$ from $x+z_0$ to $0$ of length $m-n-1\le C'\abs{\rangez}\cdot\abs{x}_1$. 
Let $C_1=(C+C')\abs{\rangez}+1$.  Then the path $x_{0:m}$ is an admissible loop that starts at $0$, goes to $x$, then takes a step to $x+z_0$, and then goes back to $0$. The path does all this in $m\le C_1\abs{x}_1$ steps.
The probability $P_0(X_{0:m}=x_{0:m})$ is bounded below by $\kappa^{C_1\abs{x}_1}$, where $\kappa=\min_{z\in\rangez}p(z)$. 

Suppose $x\in\cG_+$ is such that $V(T_x\w,z_0)\ge\e$.
Since $V(T_y\w,z)\ge0$ for all $y\in\cG_+(\rangez)$ and 
all $z\in\rangez$ we see that for any admissible loop $x_{0:n}$ 
from $0$ to $0$, 
$\exp\bigl\{-\sum_{i=0}^{n-1}V(T_{x_i}\w,z_{i+1})\bigr\}\le1$. However, for the particular path $x_{0:m}$, 
constructed in the previous paragraph, 
we have $\exp\bigl\{-\sum_{i=0}^{n-1}V(T_{x_i}\w,z_{i+1})\bigr\}\le e^{-\e}$, since $V(T_x\w,z_0)\ge\e$. Hence, 
    \begin{align*}
    E_0\Bigl[\exp \Bigl\{ - \sum_{k=0}^{\sigma_1-1}V(T_{X_k}\w,Z_{k+1})\Bigr\}  \one_{\lbrace  \sigma_1<\infty \rbrace}\Bigr]
    &\le P_0(X_{0:m}\ne x_{0:m})+e^{-\e}P_0(X_{0:m}=x_{0:m})\\
    &= 1-(1-e^{-\e})P_0(X_{0:m}=x_{0:m}).
    \end{align*}
and with the equality in \eqref{g-bound} we get
    \begin{align}\label{g-bound2}
    g(0,0)\le \kappa^{-C_1\abs{x}_1}(1-e^{-\e})^{-1}.
    \end{align}

Recall the definition of $R_\e$ in \eqref{def:R}. 
We can take $x$ in \eqref{g-bound2} with $\abs{x}_1=R_\e$ and hence we have for $r>0$
    \[\P\{\log g(0,0)\ge r\}
    \le\P\Bigl\{R_\e\ge \frac{r+\log(1-\e^{-1})}{C_1\abs{\log\kappa}}\Bigr\}.
    \]
Since we assume that $R_\e$ has $d$ finite moments we have for any $s>0$
    \begin{align*}
    \sum_{y\in\cG_+(\range)}\P\{\abs{\log g(y,y)}\ge s\abs{y}_1\}
    &\le C\sum_{r\ge0}r^{d-1}\P\{\log g(0,0)\ge sr\}\\
    &\le C\sum_{r\ge0}r^{d-1}\P\Bigl\{R_\e\ge \frac{sr+\log(1-\e^{-1})}{C\abs{\log\kappa}}\Bigr\}<\infty.
    \end{align*}
The claim of the lemma follows  from an application of Borel-Cantelli's lemma.
\end{proof}

\eqref{lim:shape:green} now follows from \eqref{g/y}, \eqref{lim:shape}, and \eqref{eq:aux1010}.\hfill\qed

\appendix

\section{Convex analysis Lemmas}\label{Appx2} 

\begin{lemma}\label{lm:aux1111}
 	Let $\range$ be a finite subset of $\R^d$ and let $\Uset$ be its convex hull. 
	If $0\in\Uset$, then let $\Usetz$ be the unique face of $\Uset$ such that $0\in\ri\Usetz$. If $0\not\in\Uset$ let $\Usetz=\varnothing$. Let $\rangez=\range\cap\Usetz$.
	Then there exist $\delta>0$ and $\widehat{u}\in \R^d$ such that $\widehat{u}\cdot z=0$ for all $z\in \rangez$ and $\widehat{u}\cdot z \ge \delta$ for all $z\in \range\setminus\rangez$.
 \end{lemma}

 \begin{proof}
	When $0\in\ri\Uset$, $\Usetz=\Uset$, $\rangez=\range$, and we can take $\widehat{u}=0$ and any $\delta>0$. The claim also holds when $0\not\in\Uset$ by applying the Hyperplane Separation Theorem.
 
 	Assume now that $0\in\Uset\setminus\ri\Uset$. Let $\cA$ be the convex hull of $\range\setminus\rangez$. 
	Let $\cA_0$ be a face of $\cA$. It is in particular a convex subset of $\Uset$. Since $\Usetz$ is a face of $\Uset$, if $\Usetz$ intersects the relative interior of $\cA_0$, then $\cA_0\subset\Usetz$ and consequently $\cA_0\cap\range\subset\rangez$. Since $\cA_0\cap\range\subset\range\setminus\rangez$ we have a contradiction and hence the relative interior of $\cA_0$ cannot intersect $\Usetz$.  Since $\cA$ is the union of the relative interiors of all its finitely many faces, we see that $\cA$ cannot intersect $\Usetz$. 
 	This implies  $\spn(\rangez)\cap \cA= \spn(\rangez)\cap\,\Uset\cap\cA=\Usetz\cap\cA=\varnothing$, where 
	$\spn(\rangez)$ is the vector space generated by $\rangez$.
 	
 	The Hyperplane Separation Theorem says that there exist $\delta, \delta'\in\R$  and $\widehat{u}\in \R^d$ such that $\widehat{u}\cdot z\le \delta'<\delta\le \widehat{u}\cdot z'$ for all $z\in \spn(\rangez) $ and $z'\in \cA.$  Since $\widehat{u}\cdot z\le \delta'$ holds for all $z\in \spn(\rangez) $, we have $\lambda\widehat{u}\cdot z\le \delta'$ for any $\lambda\in \R$. 
 	Dividing both sides by $\lambda$ and sending it to $\infty$ and $-\infty$ we get $\widehat{u}\cdot z=0.$ 
 	In particular, $\delta>0.$ We have shown that $\widehat{u}\cdot z=0$ for all $z\in \rangez$ and $\widehat{u}\cdot z \ge \delta$ for all $z\in \range\setminus\rangez.$ The lemma is proved.
 \end{proof}

\begin{lemma}\label{in-cone}
 	Let $\range$ be a finite subset of $\Q^d$. Let $\xi\in\Q^d\cap\cplus(\range).$ Then there exist rational coefficients $\gamma_{z}\ge 0$, $z\in\range$, such that $\xi=\sum_{{z}\in \range}\gamma_{z}z.$
 \end{lemma}
 
 \begin{proof}
 	There exist $\tilde{\gamma}_z\ge 0$ such that $\xi=\sum_{{z}\in \range}\tilde{\gamma}_z z$. Let $\range=\lbrace z_1,\dots,z_n\rbrace.$ If $z_1\dots,z_n$ are not linearly independent then we can write $\sum_{i=1}^{n}a_iz_i=0$ such that not all $a_i$ are zero. Then for any $t \in \R$,
	\[\xi
	=\sum_{i=1}^n\left(\tilde{\gamma}_{z_i}-t a_i\right)z_i. \] 
 	Let \[t=\max \Bigl( \min\Bigl\{\frac{\tilde{\gamma}_{z_i}}{a_i}:a_i>0 \Bigr\},  \max\Bigl\{\frac{\tilde{\gamma}_{z_i}}{a_i}:a_i<0 \Bigr\}  \Bigr).
 	\]
 	
 	If $t<0,$ then it must equal  $ \max\left\lbrace {\tilde{\gamma}_{z_i}/a_i}:a_i<0 \right \rbrace $ and hence for all $i\in\lbrace1,\dots,n\rbrace,$ either $a_i=0$ or $a_i<0$ and $t\ge \tilde{\gamma}_{z_i}/a_i$. In either case, $\tilde{\gamma}_{z_i}-t a_i\ge 0.$ Furthermore, equality is achieved for at least one $i.$
 	
 	If, on the other hand, $t\ge 0$, then $ t= \min\left\lbrace {\tilde{\gamma}_{z_i}/a_i}:a_i>0 \right \rbrace$. In this case, for each $i\in\lbrace1,\dots,n\rbrace$, either $a_i\le 0$ or $a_i>0$ and $t\le \tilde{\gamma}_{z_i}/a_i$. In both cases, $\tilde{\gamma}_{z_i}-t a_i\ge 0$. Again,  equality is achieved for some $i.$
 	
 	In either case, with our choice of $t$ we have $\tilde{\gamma}_{z_i}-t a_i\ge 0$ for all $i\in \lbrace 1,\dots,n\rbrace$ and there exists an $i$ such that $\tilde{ \gamma}_{z_i}-ta_i=0$.  After discarding all $z_i$ for which equality holds, and reindexing, we end up with 
 	\[ \xi = \sum_{i=1}^k \left(\tilde{\gamma}_{z_i}-ta_i  \right)z_i,\]
 	with $k\le n-1.$
 	We can assume $\lbrace z_1,\dots,z_k \rbrace$ are linearly independent because otherwise we can again write $\sum_{i=1}^k 
 	b_i z_i=0$ and repeat the above procedure, every time reducing the number of $z$'s by at least one. Define 
 	$\gamma_{z_i}=\tilde{\gamma}_{z_i}-ta_i.$ Let 
 	$\mathbf{x}$ be the column vector with entries $\gamma_{z_i}$, $i\in\lbrace1,\dots,k\rbrace$, and let 
	 $A=\left[z_1 \cdots z_k  \right]$ be the $d\times k$ matrix with column vectors $z_i$.
 	Then $A$ has full rank and rational entries and $A\mathbf{x}=\xi$ implies $\mathbf{x}=(A^T A)^{-1}A^T\xi$. This tells us $\mathbf{x}\in\Q^d.$  
\end{proof}

 \begin{proof}[Proof of Lemma \ref{boundednessofsteps}]  
 	Recall that $\Usetz$ is the unique face of $\Uset$ such that $0\in \ri \Usetz$ and $\rangez=\range\cap\Usetz$, with the convention that $\Usetz=\rangez=\varnothing$ if $0\not \in \Uset$. By Lemma \ref{lm:aux1111} there exist $\widehat{u}\in\R^d$ and $\delta>0$ such that $z\cdot \widehat{u}\ge \delta$ for $z\in \range\setminus\rangez$ and $z\cdot\widehat{u}=0$ for $z\in \rangez$. 
	
	Denote the dimension of the linear span of $\rangez$ by $m$. By the Fundamental Theorem of Lattices 
	\cite[Lemma 3.4]{Tao-Vu-10} there exist linearly independent vectors $\lbrace z_1,\dotsc,z_m \rbrace\subset \rangez$ that at the same time form a basis of $\spn(\rangez)$ and also generate the group $\cG(\rangez)$. 
%

 	For $\xi\in \cplus(\range)$ we can write $\xi =\sum_{{z}\in \range}\overline{ \gamma}_z(\xi)z$ with $\overline{ \gamma}_z(\xi)\ge 0.$ If $\xi \in \gplus(\range)$ we can assume $\overline{ \gamma}_z(\xi)\in \Z_+$ for all $z\in \range.$ Let $\gamma_{z}(\xi)=\overline{ \gamma}_z(\xi)$ for $z\in \range\setminus\rangez$. Then 
	\begin{align}\label{gamz1}
	0\le \gamma_{z}(\xi)\le \delta^{-1}\!\!\!\sum_{\overline{z}\in \range\setminus\rangez}\gamma_{\overline{z}}(\xi)\, \overline{ z}\cdot\widehat{u} =\delta^{-1}\sum_{\overline{z}\in \range}\overline{\gamma}_{\overline{z}}(\xi) \overline{ z}\cdot\widehat{u}=\delta^{-1}\xi\cdot\widehat{u}\le \delta^{-1}|\widehat{u}|_\infty|\xi|_1
	\end{align}
	for all $z\in\range\setminus\rangez$
	and
	\begin{align}\label{gamz2}
	\text{if $\xi\in\gplus(\range)$, then $\gamma_z(\xi)\in\Z_+$ for all $z\in\range\setminus\rangez$.}
	\end{align}
 	
 	Let $\xi'=\sum_{{z}\in \rangez}\overline{ \gamma}_z(\xi)z\in\spn(\rangez)$. 
 	Then there exist unique  $\tilde{ \gamma}_{z_i}(\xi')\in\R$, $i\in \lbrace 1,\ldots,m\rbrace,$ such that $\xi'=\sum_{i=1}^m\tilde{ \gamma}_{z_i}(\xi') z_i$. If $\xi\in \gplus(\range)$, then $\xi'\in \gplus(\rangez) $ and  $\tilde{ \gamma}_{z_i}(\xi')\in \Z$.  These functions are linear in $\xi'$ and hence there exists a constant $C$ such that 
 	\begin{align*}
 	|\tilde{ \gamma}_{z_i}(\xi')|\le C|\xi'|_1\le&C\Bigl( |\xi|_1+\sum_{{z}\in \range\setminus \rangez}{ \gamma}_z(\xi)|z|_1 \Bigr)
 	\le C\Bigl( 1+\delta^{-1} |\widehat{u}|_\infty \sum_{{z}\in \range\setminus \rangez}|z|_1 \Bigr)|\xi|_1.
 	\end{align*}
 	By Corollary A.3.\ of \cite{Ras-Sep-Yil-13}, for each $i\in \lbrace 1,\ldots, m\rbrace$ there exist $b_i(z)\in \Z_+$, $z\in \rangez$, such that $-z_i=\sum_{{z}\in \rangez}b_i(z)z.$ For $z\in\rangez$ set 
 	
 	\begin{align}\label{gamz}
	 \gamma_{z}(\xi) =\sum_{i=1}^m \Bigl[ (\tilde{ \gamma}_{z_i}(\xi'))^+\one{\lbrace  z=z_i \rbrace} +(\tilde{ \gamma}_{z_i}(\xi'))^-b_i(z) \Bigr].
	 \end{align}
 	Then 
 	\begin{equation}\label{gamz3}
 	0\le \gamma_{z}(\xi)\le \Bigl[1+\sum_{i=1}^m b_i(z) \Bigr]  \max_{1\le i\le m}|\tilde{ \gamma}_{z_i}(\xi')|\le  C'|\xi|_1\quad\text{for all $z\in\rangez$}.
 	\end{equation}
	Also \eqref{gamz} shows that
	\begin{align}\label{gamz4}
	\text{if $\xi\in\gplus(\range)$, then $\gamma_z(\xi)\in\Z_+$ for all $z\in\rangez$.}
	\end{align}
 	Lastly, 
 	\begin{align*}
 	\xi'=&\sum_{i=1}^m\tilde{ \gamma}_{z_i}(\xi')z_i\\
 	=&\sum_{i=1}^m\left(\tilde{ \gamma}_{z_i}(\xi')\right)^+ z_i+\sum_{i=1}^m\left(\tilde{ \gamma}_{z_i}(\xi')\right)^- (-z_i)\\
 	=&\sum_{i=1}^m \sum_{{z}\in  \rangez} \left(\tilde{ \gamma}_{z_i}(\xi')\right)^+  \one{\lbrace  z=z_i \rbrace} z+\sum_{i=1}^m \sum_{{z}\in \rangez} \left(\tilde{ \gamma}_{z_i}(\xi')\right)^- b_i(z)z \\
 	=&\sum_{{z}\in \rangez} \gamma_{z}(\xi)z
 	\end{align*}
 	and thus 
		\[\xi=\sum_{z\in\range}\overline\gamma_z(\xi)z=\xi'+\sum_{z\in\range\setminus\rangez}\gamma_z(\xi) z=\sum_{{z}\in \range}\gamma_{z}(\xi)z.\]
		This, together with \eqref{gamz1}, \eqref{gamz2}, \eqref{gamz3}, and \eqref{gamz4} show that
		 coefficients $\gamma_z(\xi)$ satisfy all the claims of the lemma.
 \end{proof}

\section{The proof of Theorem \ref{th:Gshape}}\label{app:Gshape}
We prove the shape theorem for $G_{0,(n),x}$. This can be repeated almost word-for-word to produce the proof of the shape theorem for $G^\infty_{0,(n),x}$.\smallskip

For $m,n\in\N$ and $x,y,z\in \Z^d$ such that 
$y-x\in D_m$ and $ z-y\in D_n$, 
decomposing into the values of $X_m$ we can write
 \[G_{x,(m+n),z}=\log\sum_{\substack{v:v-x\in D_m\\\ \ z-v\in D_n}} e^{G_{x,(m),v}}\cdot e^{G_{v,(n),z}}. \] 
This implies the superadditivity
\begin{equation}\label{supmult}
G_{x,(m+n),z}\ge G_{x,(m),y}+G_{y,(n),z}.
\end{equation}

	The proof of Theorem \ref{th:Gshape} proceeds by  contradiction. Fix $\delta>0.$  Assume the shape theorem does not hold. Thus, with positive probability  there exists an $\e >0$ and  a sequence $x_{\ell} \in \ell\Usetgen_\delta\cap  D_\ell $ such that $\ell\to \infty$ and 
	\begin{equation}
  \dfrac{\abs{G_{0,(\ell),x_{\ell} }-\ell\Lapp\left(\frac{x_{\ell} }{\ell}\right) }}{\ell}\ge \e.
	\end{equation}
	
	Let $\Omega'_1$ be the intersection of the above event with the full-measure events on which
	\eqref{eq:p2p} is satisfied and
	\eqref{def:cL} holds with $g(T_{x+kz}\w)=V(T_{x+kz}\w,z)$, for each $z\in\rangegen\setminus\{0\}$, and with $g(T_{x+k\hat z}\w)=V(T_{x+k\hat z}\w,0)$ for $\hat{z}$ as in the statement. We now work with a fixed $\w$ from $\Omega'_1$.
	
	We have $x_{\ell} =\sum_{z\in \rangegen}b_{\ell,z} z$ with $ b_{\ell,z}\in \Z_+$ and $\sum_{{z}\in \rangegen}b_{\ell,z} =\ell$.
	By compactness, we can find a subsequence $\ell_n$ and $\gamma_z\in [0,1]$ such that $b_{\ell_n,z}/\ell_n\to \gamma_z$ for all $z\in \rangegen.$ Then $x_{\ell_n}/\ell_n\to \xi \in \Usetgen_\delta$ where $\xi=\sum_{z\in \rangegen}\gamma_z z$ and $\sum_{z\in \rangegen}\gamma_z =1.$ Abbreviate $\ell_n$ by writing just $n.$  Choose some large $N$ such that   $\abs{x_n/n-\xi}_1<\e$ and 
	\[ \Bigl| \Lapp\Bigl(\frac{x_n}{n}\Bigr)-\Lapp(\xi) \Bigr| <\e/2 \]   
	for $n>N.$ Here we used the continuity of $\Lapp$ on $\Usetgen_\delta$. Then, for $n>N$ we have 
		\begin{equation}\label{contradiction assumption}
 \Bigl| \dfrac{ G_{0,(n),x_{n} }}{n} -\Lapp(\xi)\Bigr|\ge \e/2.
	\end{equation}
	Let $\e_1>0.$ 
	Let $\overline{\kappa}_{m,z}=\ce{m(\gamma_{z}+\e_1)}/\overline{r}_m$ where $\overline{r}_m=\sum_{{z}\in \range'} \ce{m(\gamma_z+\e_1)}$ with $m$ large, to be chosen further down. Note that $\overline{r}_m/m\to 1+\e_1\abs{\range'}$ as $m\to \infty$. Note also that 
	 \[\overline{\kappa}_{m,z}\to \frac{\gamma_{z}+\e_1}{1+\abs{\range'}\e_1}\quad \text{as}\quad m\to \infty. \]
	 Fix $m$ large enough so that for all $z\in \range'$
	 \begin{equation}\label{bounds for kappamz}
	 \frac{\gamma_{z}+\e_1/2}{1+\abs{\range'}\e_1} < \overline{\kappa}_{m,z} < \frac{(\gamma_{z}+2\e_1)}{1+\abs{\range'}\e_1}.
	 \end{equation}
	  Let   
	 \begin{equation*}
	 \overline{ \zeta}_m=\sum_{{z}\in \range'}\overline{\kappa}_{m,z} z, \quad \overline{k}_n = \Bigl\lfloor \frac{(1+\abs{\range'}\e_1)n}{\overline{r}_m}\Bigr\rfloor, \quad \text{and}\quad \overline{s}_z^{(n)}=\overline{r}_m	\overline{k}_n\overline{\kappa}_{m,z}-b_{n,z}.
	 \end{equation*}
	 Then  for any $z\in\range'$
	\begin{equation}\label{lower bound for s_z^n /n}
\frac{\overline{ s}_z^{(n)}}{n}\to(1+\abs{\range'}\e_1)\overline{\kappa}_{m,z}-\gamma_{z} > \e_1/2>0 \quad \text{as } n\to \infty
	\end{equation}

	 Thus, $\overline{s}_z^{(n)}> 0$ for large enough $n$
	 and then $\overline{r}_m \overline{k}_n \overline{ \zeta}_m-  x_n= \sum_{z\in \range'} (\overline{r}_m \overline{k}_n  \overline{\kappa}_{m,z} -b_{n,z}) z \in D_{\overline r_m\overline k_n-n}.$ 
	By using \eqref{supmult}, we get
	\begin{equation}\label{eq3 of shape theorem of restricted path}
G_{0,(n),x_n } \le G_{0,( \overline{r}_m\overline{k}_n),  \overline{r}_m\overline{k}_n\overline{ \zeta}_m }-G_{x_n,( \overline{r}_m \overline{k}_n-n), \overline{r}_m\overline{k}_n\overline{ \zeta}_m}. 
	\end{equation}
Similarly, let  $\underline{\kappa}_{m,z}=\fl{m\gamma_{z}}/\underline{r}_m$ where $\underline{r}_m=\sum_{{z}\in \range'}\fl{m\gamma_{z}}$. Note that if $\gamma_{z}=0$ then $\underline{\kappa}_{m,z}=0.$ Also, $\underline{r}_m/m\to 1$ and $\underline{\kappa}_{m,z}\to \gamma_z$ as $m\to \infty.$ Let $\delta'= \min_{ z\in \range' }\gamma_{z} >  0$.  Fix $m$ large enough such that 
$\underline{\kappa}_{m,z}\in [\delta'/2,1]$ 
 and $\abs{\underline{\kappa}_{m,z}-\gamma_{z}}<\e_1$  for all $z\in \range'.$  Now, suppose $\e_1<\delta'/\left(2\abs{\range'}\right)$ and let
\begin{equation*}
\underline{\zeta}_m=\sum_{{z}\in \range'}\underline{\kappa}_{m,z} z,\quad\underline{k}_n =  \Bigl\lfloor \frac{(1-2|\range'|\e_1/\delta')n}{\underline{r}_m}\Bigr\rfloor\quad \text{and}\quad \underline{s}_z^{(n)}=	b_{n,z}- \underline{r}_m\underline{k}_n\underline{\kappa}_{m,z} .
\end{equation*} 
Thus 
\begin{align*}
\frac{\underline{s}_z^{(n)}}{n}&\mathop{\longrightarrow}_{n\to\infty}\gamma_{z}- (1-2\abs{\range'}\e_1/\delta')\underline{\kappa}_{m,z}
\ge \gamma_{z}-\underline{\kappa}_{m,z}+|\range'|\e_1
\ge -\e_1+|\range'|\e_1 \ge\e_1>0.
\end{align*}
Then for $n$ large  $x_n-\underline{r}_m\underline{k}_n \underline{\zeta}_m = \sum_{z\in \range'} (b_{n,z}- \underline{r}_m\underline{k}_n\underline{\kappa}_{m,z} )z \in D_{n-\underline r_m\underline k_n}$. 
By using \eqref{supmult}, we get
\begin{equation}\label{eq4 of restricted path}
G_{0,(\underline{r}_m\underline{k}_n),\underline{r}_m\underline{k}_n\underline{\zeta}_m }+ G_{\underline{r}_m\underline{k}_n\underline{\zeta}_m,(n-\underline{r}_m\underline{k}_n),x_n }\le G_{0,(n),x_n }. 
\end{equation}
	Since  
	\begin{equation}\label{limit of over and under rk mn's divided n}
 \frac{ \overline{r}_m \overline{k}_n}{n}\to 	 1+|\range'|\e_1 \quad 
\text{and}\quad  \frac{\underline{r}_m\underline{k}_n }{n}\to 1-2|\range'|\e_1/\delta' \quad \text{as }  n\to \infty.
	\end{equation}
we have that for large $n$, if $z\in \range'$ then both  $\overline{r}_m \overline{k}_n   \overline{\kappa}_{m,z} -b_{n,z}$ and $b_{n,z}-\underline{r}_m \underline{k}_n  \underline{\kappa}_{m,z} $ 
	are bounded above by 
	\begin{align*}
	\overline{r}_m \overline{k}_n  \overline{\kappa}_{m,z}- \underline{r}_m \underline{k}_n    \underline{\kappa}_{m,z}
	&\le  n( (1+\abs{\range'}\e_1)+\e_1)\overline{\kappa}_{m,z}- n\left((1-2\abs{\range'}\e_1/\delta')-\e_1\right)\underline{\kappa}_{m,z}  \\ 
    &\le n( (1+\abs{\range'}\e_1)+\e_1)\frac{\gamma_{z}+2\e_1}{1+\abs{\range'}\e_1}- n\bigl((1-2\abs{\range'}\e_1/\delta')-\e_1\bigr)(\gamma_{z}-\e_1)  \\ 
    &=n\e_1\Bigl(3+\frac{\gamma_{z}+2\e_1}{1+\abs{\range'}\e_1}+2\abs{\range'}(\gamma_{z}-\e_1)/\delta'+(\gamma_{z}-\e_1)\Bigr)  \\
    &\le \bigl(6+2\abs{\range'}/\delta'\bigr)n\e_1=c_1n\e_1.  
	\end{align*}
	
	For $z\in \range\setminus\range'$, $\gamma_{z}=0$, $\underline{\kappa}_{m,z}=0$, and we have 
	\begin{align*}
	0\le 	\overline{r}_m \overline{k}_n  \overline{\kappa}_{m,z}- \underline{r}_m \underline{k}_n    \underline{\kappa}_{m,z}
	&\le n\bigl((1+\abs{\range'}\e_1)+\e_1\bigr)\frac{2\e_1}{1+\abs{\range'}\e_1}
	\le 3n\e_1\le c_1n\e_1.
	\end{align*}
	 We next develop, a lower bound for $G_{x_n,(	\overline{r}_m \overline{k}_n-n),	\overline{r}_m\overline{k}_n\overline{ \zeta}_m}.$  Fix a path from $x_n$ to $\overline{r}_m\overline{k}_n\overline{ \zeta}_m$	that takes $\overline{ s}_z^{(n)}$ $z$-steps for each $z\in \range'$.
	Since $\Usetgen$ is not a singleton there exists a nonzero $\hat{z}\in \range'$. If $0\in\range'$ then we have by \eqref{bounds for kappamz}, \eqref{lower bound for s_z^n /n}, and \eqref{limit of over and under rk mn's divided n}, that for a large enough $n$ 
	\begin{align*}
\frac{\overline{s}_0^{(n)}}{\overline{ s}_{\hat{z}}^{(n)}}
\le \frac{n\bigl((1+\abs{\range'}\e_1)\overline{\kappa}_{m,0}-\gamma_{0}+\e_1\bigr)}{n\e_1/2} 
\le \frac{\gamma_0+2\e_1-\gamma_0+\e_1}{\e_1/2}  =6. 
	\end{align*}
	This tells us the ratio of zero steps to $\hat{z}$ steps is at most $6$. Rearrange the path as follows.  Start the path  with blocks of $\hat{z}$-steps  followed by at most $6$ zero steps, until $\hat{z}$-steps and zero-steps exhausted. After that fix an ordering of  $\range'\setminus\lbrace0,\hat{z}\rbrace=\lbrace z_1, z_2,\ldots\rbrace$ and arrange the rest of the path  to take first all its $z_1$-steps then all $z_2$-steps and so on. Also note that any point $y$ on the path is such that $y\in \gplus(\range')$ and 
	\[ \abs{y}_1\le \abs{x_n}_1+(\overline{ r}_m\overline{ k}_n-n)\max_{ z\in \range' }\abs{z}_1 \le n(1+\abs{\range'}\e_1)\max_{ z\in \range' }\abs{z}_1=c_2n. \]
	Thus 
\begin{align*}
G_{x_n,(\overline{r}_m \overline{k}_n-n),\overline{r}_m\overline{k}_n\overline{ \zeta}_m} 
&\ge 
- \abs{\rangegen}\max_{ \substack{y\in\gplus(\rangegen) \\ \abs{y}_1\le c_2 n }}\max_{z\in \rangegen\setminus \lbrace 0\rbrace} \sum_{0\le i\le c_1\e_1 n} \abs{V(T_{y+iz}\w,z) }\\
&\qquad- 6\max_{ \substack{y\in\gplus(\rangegen) \\ \abs{y}_1\le c_2 n }} \sum_{0\le i\le c_1\e_1 n} \abs{V(T_{y+i\hat z}\w,0) }
+c_1\e_1 n\min_{ z\in \range'}\log p(z). 
	\end{align*}
	Now, by dividing both sides  by $n$ and taking $n \to \infty$ we get that 
	\begin{align*}
	\varliminf_{n \to \infty} \dfrac{G_{x_n,(\overline{r}_m \overline{k}_n-n),\overline{r}_m\overline{k}_n\overline{ \zeta}_m}}{n} 
	&\ge
	- \abs{\range'} \varlimsup_{n \to \infty} \max_{\substack{y\in\gplus(\range') \\ \abs{y}_1\le c_2 n }}\max_{z\in \range'\setminus \lbrace 0\rbrace} \frac{1}{n} \sum_{0\le i\le c_1 \e_1 n} \abs{V(T_{y+iz}\w,z) }\\
	&\qquad- 6 \varlimsup_{n \to \infty} \max_{\substack{y\in\gplus(\range') \\ \abs{y}\le c_2 n }}\frac{1}{n} \sum_{0\le i\le c_1 \e_1 n} \abs{V(T_{y+i\hat z}\w,0) }
	+c_1\e_1 \min_{ z\in \range'}\log p(z). 
	\end{align*}
	Fix any $\e_2>0.$ Since $\w\in\Omega_1'$ we can find $\e_1$ small enough so that the right-hand side in the above display is bounded below by $- \e_2$. 
	Similarly, we see that 
	\[
	\varliminf_{n\to \infty} \dfrac{G_{\underline{r}_m\underline{k}_n\underline{\zeta}_m ,(n- \underline{r}_m\underline{k}_n),x_n}}{n}\ge- \e_2. \] 
	These two bounds, together with  \eqref{eq3 of shape theorem of restricted path}, \eqref{eq4 of restricted path}, \eqref{limit of over and under rk mn's divided n}, and \eqref{eq:p2p}  give
	\begin{equation*}
	-\e_2+  (1-2\abs{\range'}\e_1/\delta')  \Lapp\bigl(\underline{\zeta}_m \bigr) \le \varliminf_{n\to\infty}  \frac{ G_{0,(n),x_{n} }}{n}\le \varlimsup_{n\to\infty}  \frac{ G_{0,(n),x_{n} }}{n}\le   (1+\abs{\range'}\e_1)  \Lapp\bigl(\overline{\zeta}_m\bigr)+\e_2. 
	\end{equation*}
	Since $\xi$ is on the face $\Usetgen$, $z\in \Usetgen$ for all $z\in \range'$ and both  $\underline{\zeta}_m$ and $\overline{\zeta}_m$  are in $\Usetgen$. Since $\Lapp$ is continuous on $\ri \Usetgen$  and $\xi\in\ri \Usetgen$, we have for $\e_1>0$ small enough
	\[ \Lapp\bigl(\underline{\zeta}_m\bigr)\to \Lapp(\xi) \quad \text{and} \quad \Lapp\bigl(\overline{\zeta}_m\bigr) \to \Lapp\Bigl((1+\abs{\range'}\e_1)^{-1}\Bigl(\xi +\e_1\sum_{{z}\in \range'}z\Bigr)\Bigr) \quad  \text{as} \quad m\to \infty.\]
	Take $m\to \infty$ then   $\e_1\to 0,$ use again the continuity of $\Lapp$ on $\ri \Usetz,$  and finally take $\e_2 \to 0$  to get that 
	\[\lim_{n\to \infty}\dfrac{G_{0,(n),x_{n} }}{n}=\Lapp(\xi),
	\]
	which contradicts \eqref{contradiction assumption}. Theorem \ref{th:Gshape} is proved.	\hfill\qed
\bibliographystyle{plain}
\bibliography{Lyapunov-AIHP}

\end{document}